\documentclass [12pt]{article}
\usepackage{amsmath,amssymb}
\usepackage{amssymb,verbatim}

\usepackage{blkarray}
\usepackage{amsmath}
\usepackage{tikz-cd} 

\usepackage{faktor}
\usepackage{xfrac}
\usepackage{faktor} 
\usepackage{amsmath} 
\usepackage{amssymb}
\usepackage{mathtools}

\usepackage[utf8]{inputenc}
\usepackage{graphicx}
\usepackage{graphicx,amssymb,amstext,amsmath}
\usepackage{tikz}
\usepackage{amsmath}
\usepackage{amssymb}
\usepackage{amsthm}
\usepackage{hebcal}
\usetikzlibrary{arrows, decorations.markings}
\usepackage{blindtext}
\usepackage{amsthm}
\usepackage[T1]{fontenc}
\usepackage[utf8]{inputenc}
\usepackage{blindtext}
\usepackage[T1]{fontenc}
\usepackage[utf8]{inputenc}

\usepackage{enumitem}
\usepackage{multirow}
\usepackage{array}
\newcolumntype{L}[1]{>{\raggedright\let\newline\\\arraybackslash\hspace{0pt}}m{#1}}
\newcolumntype{C}[1]{>{\centering\let\newline\\\arraybackslash\hspace{0pt}}m{#1}}
\newcolumntype{R}[1]{>{\raggedleft\let\newline\\\arraybackslash\hspace{0pt}}m{#1}}
\usepackage{tikz}
\usetikzlibrary{positioning, arrows.meta}
\newcommand{\here}[2]{\tikz[remember picture]{\node[inner sep=0](#2){#1}}}

\usepackage{cite}
\newtheorem{thm}{Theorem}[section]
\newtheorem{lemma}[thm]{Lemma}
\newtheorem{prop}[thm]{Proposition}

\newtheorem{definition}{Definition}[section]

\newtheorem{example}[thm]{Example}
\newtheorem{remark}[thm]{Remark}
\newtheorem{question}[thm]{Question}
\newtheorem{conjecture}[thm]{Conjecture}

\newtheorem{question
}{Question}
\newtheorem*{acknowledgement}{Acknowledgement}
\def\0{{\bf 0}}

\def\N{{\bf N}}

\def\R{{\bf R}}
\def\Z{{\bf Z}}

\def\SD{\mathop{\rm SD}\nolimits}

\makeatletter
\def\keywords{\xdef\@thefnmark{}\@footnotetext}
\title{Simultaneous Visibility in the Integer Lattice}
\date{}

\author{Daniel Berend\footnote{
Departments of Mathematics and Computer Science, Ben-Gurion
University, Beer Sheva 84105, Israel.
E-mail: berend@math.bgu.ac.il}
\footnote{Research supported in part by
the Milken Families Foundation Chair in
Mathematics.}
\and
Rishi Kumar\footnote{Department of Mathematics, Ben-Gurion
University, Beer Sheva 84105, Israel.
E-mail: kumarr@post.bgu.ac.il}
\and
Andrew Pollington\footnote{
National Science Foundation, Arlington,  VA 22230, USA.
E-mail: adpollin@nsf.gov}
}
\begin{document}
\maketitle
\begin{abstract}
Two lattice points are visible from one another if there is no lattice point on the open line segment joining them. Let $S$ be a finite subset of $\Z^k$. The asymptotic density of the set of lattice points, visible from all points of $S$, was studied by several authors. 
Our main result is an improved upper bound on the error term. We also find the Schnirelmann density of the set of visible points from some sets $S$. Finally, we discuss these questions from the point of view of ergodic theory.
\end{abstract}

\section{Introduction and Statement of Results}\label{sec, intro}
Let $\Z^k$ be the $k$-dimensional integer lattice, $k\geq 2$. 
Two distinct points $\textbf{x}$ = $(x_1,\ldots ,x_k)$ and $\textbf{y}$ = $(y_1, \ldots ,y_k)$ in $\Z^k$ are \textit{mutually visible} if no other point of $\Z^k$ lies on the line segment joining them.
It is easily seen~\cite{Rea} that $\textbf{x}$ and $\textbf{y}$ are mutually visible if and only if $\gcd(x_1- y_1, \ldots, x_k -y_k) =1$. In this paper,
we deal with sets defined by various visibility conditions. The questions belong to a big body of questions relating to lattice points.

\subsection{Visibility from the Origin}
Dirichlet considered the ``size'' of the set $V_2$ of points visible from the origin in $\N^2$.
More precisely, the \textit{asymptotic density} of a set $A\subseteq \N^k$, denoted by $D(A)$, is defined by
$$D(A) = \lim_{L\to \infty} \frac{\left|A\cap [1,L]^k\right|}{L^k},$$
provided the limit exists.
Dirichlet showed that $D(V_2) = 1/\zeta(2)$, where $\zeta$ is the Riemann zeta function. In fact, Dirichlet's result was stated in terms of the function $\Phi(L)= \sum_{n=1}^L \varphi(n)$, where $\varphi$ is Euler's totient function, directly related to our problem since
$$\left| V_2\cap [1,L]^2\right|= 2\Phi(L)-1.$$
Denote:
$$E(L)=\Phi(L) -\frac{3}{\pi^2}L^2.$$
Dirichlet showed that the error $E(L)$ is bounded by $O(L^{\delta})$ for some $1<\delta <2$~(see, for example, \cite{ERDOS}). The error term has been improved later \cite{MERF, WALFISZ}, and the currently best known bound, due to Liu~\cite{LIU}, is $O\left(L(\log L)^{2/3} (\log \log L)^{1/3}\right)$.
Pillai and Chowla~\cite{PILLAI} showed that, on the other hand, 
$$E(L)= \Omega(L \log \log \log L),$$
and 
$$\sum_{n=1}^LE(n)\sim \frac{3}{2\pi^2}L^2.$$
Sylvester~\cite{SYLVE1,SYLVE2} conjectured that $E(L)>0$ for all positive $L$. However, Sarma~\cite{SARMA} observed that $E(820)<0$, and
Erd\H{o}s and Shapiro~\cite{ERDOS} proved that $E(L)$ changes sign infinitely often and, moreover,
$$E(L)=  \Omega_{\pm}(L \log \log \log \log L).$$
Currently, the best-known estimate in this direction, due to Montgomery~\cite{MONT}, is
\begin{equation}\label{intro, montgomery}
E(L) =  \Omega_{\pm}(L \sqrt{\log \log L}).
\end{equation}
Montgomery also conjectured that $E(L) = O(L \log \log L)$ and $E(L) =\Omega_{\pm}(L \log \log L)$.

Lehmer~\cite{DNL} extended Dirichlet's result to any dimension $k\geq 3$, showing that the asymptotic density of the set of points of $\N^k$, visible from the origin, is $1/\zeta(k)$. Nymann~\cite{Nym} bounded the error in this case by $O(L^{k-1})$. We note, however, that the proportion of visible points is not close to the asymptotic density for all large cubes. In fact, by a simple use of the Chinese Remainder Theorem, we can see that there exist arbitrarily large cubes in $\Z^k$, containing no points visible from $\textbf{0}\in \Z^k$ (see~\cite[Theorem 5.29]{Apo}).

\subsection{Simultaneous Visibility}
Given a set $S\subseteq \Z^k$, denote by $V(S)$ the set of points of $\Z^k$, visible simultaneously from all points of $S$.
The set $S$ is \textit{admissible} if every two points in $S$ are mutually visible. Let $S$ be an admissible set of cardinality $r$ in $\Z^k$.
Rearick~\cite{ReaM} showed that $V(S,[1,L]^k)= |V(S)\cap [1,L]^k|$ is given by
\begin{equation}\label{admi}
V(S,[1,L]^k) = L^k\prod_{p\in \mathcal{P}}\left(1- \frac{r}{p^k}\right) + E(L), \end{equation}
where $\mathcal{P}$ is the set of all primes and the error term $E(L)$ satisfies:
\begin{equation}\label{Reerror}
E(L) = \begin{cases} 
 O(L^{k-1}), & \qquad   r< k-1, \\
O(L^{k- \frac{k-1}{r}+\varepsilon}), ~\forall~\varepsilon > 0, & \qquad r\geq k-1.
\end{cases}
\end{equation}

Liu, Lu, and Meng \cite{LU} considered simultaneous visibility along curves. In some cases, their results, restricted to visibility along straight lines, improve Rearick's bound on the error (see also \cite{CHAUBEY}).

Rumsey~\cite{Rum} dealt with the case where $S$ is an arbitrary subset of $\Z^k$. Let 
$$\pi_p: \Z^k \to (\Z/ p\Z)^k, \qquad p\in \mathcal{P},$$ 
be the natural projection. Put:
\begin{equation}\label{S(p)}
s(p) = |\pi_p(S)|, \qquad p\in \mathcal{P}.
\end{equation}
In the case when $S$ is finite set, Rumsey~\cite{Rum} showed that 
\begin{equation}\label{s(p)}
D(V(S)) = \prod_{p\in \mathcal{P}}\left(1 - \frac{s(p)}{p^k}\right).
\end{equation}
He also generalized this result to the case of infinite sets $S$, satisfying appropriate conditions.

We continue with the same setup, and consider simultaneous visibility from any finite set of lattice points, not necessarily admissible. Unlike all previous studies, we do not confine ourselves to the set of visible points in a large cube $[1, L]^k$ starting at the origin. Rather, we consider visible points in any large box located anywhere. Thus, let
\begin{equation}\label{box}
B =J_1\times J_2 \times \cdots \times J_k,
\end{equation} 
be a box in $\Z^k$, where $J_i= [M_i,M_i +L_i)$ for some integers $M_i$ and $L_i$ for $1\leq i \leq k$. We may assume, without loss of generality, that $L_1\geq \ldots \geq L_k$.
\begin{thm}\label{main thm}
Let $S$ be a finite subset of $\Z^k$ of cardinality $r$, and let $B$ be as in \eqref{box} and $s(p)$ as in \eqref{S(p)}. Then the number of points of $B$, visible from $S$, satisfies, as $L_k = \min\{L_1,\ldots, L_k\}\to \infty$,
\begin{equation}\label{section 2 , formula for main result}
V(S,B) \leq L_1\cdots L_k\prod_{p \in \mathcal{P}}\left(1 - \frac{s(p)}{p^k}\right) +  E,
\end{equation}
where
\begin{equation}\label{main, theorem error term in section 2}
 E = \begin{cases} 
 O\left(\max\{L_1\log^{3r}L_2, (L_1L_2)^{2/3+\varepsilon}\}\right),\forall \varepsilon>0, &\qquad  k = 2,\\
     O(L_1\cdots L_{k-1}), & \qquad   k\geq 3.\\
    \end{cases}
\end{equation}
In particular, if $L_1=L_2=\ldots =L_k= L$, then
$$V(S,B) \leq L^k\prod_{p \in \mathcal{P}}\left(1 - \frac{s(p)}{p^k}\right) + E,$$
where
\begin{equation*}
 E = \begin{cases} 
 O\left(L^{4/3+\varepsilon}\right),\forall \varepsilon>0, & \qquad k = 2,\\
     O(L^{k-1}), &\qquad    k\geq 3.\\
    \end{cases}
 \end{equation*}
 \end{thm}

 \begin{remark}\emph{
 For $k=2$, if $L_2$ is not very small relative to $L_1$, more precisely if $L_2\geq ~L_1^{1/2}$, the maximum in the first line of \eqref{main, theorem error term in section 2} is attained by the second term, so that $E= O\left((L_1L_2)^{2/3 + \varepsilon}\right)$. If $L_2$ is smaller, then  $E= O\left(L_1\log ^{3r}L_2\right)$.}
 \end{remark}
 
 Note that, in \eqref{section 2 , formula for main result}, we have an inequality in one direction only. As mentioned above, $V(S, B)$ may even vanish for arbitrary large $L_i$-s.
 Comparing the upper bound in Theorem~\ref{main thm} with those of Rearick \cite{ReaM} and Liu et al. \cite{LU} (for admissible sets $S$ of cardinality at least 2), where their results apply, we see that our results are better in some cases and equally good in others. The improvements are due to the method used. We use here the higher-dimensional Selberg sieve, whereas former papers used elementary methods.

 As mentioned above, \eqref{section 2 , formula for main result} cannot possibly have a counterpart with the direction of the inequality reversed, as there are always arbitrarily large boxes $B$ for which $V(S, B)=0$. We do believe, however, that \eqref{section 2 , formula for main result} holds when inverting the direction of the inequality for ``most'' boxes. In particular, it seems plausible that the reverse inequality holds for cubes of the form $B=[1, L]^k$.
 
 For $k\geq 3$ , $S=\{(0,\ldots,0)\}$ and $B= [1,L]^k$, it readily follows from Takeda~\cite{TAKEDA} that $E= \Omega(L^{k-1})$. We will show that the error may be both positive and negative.
 
\subsection{Schnirelmann Density}
One may also be interested in the Schnirelmann density of the set $V(S)$. Recall that the \textit{Schnirelmann density} of a set $A\subseteq \N$ is given by:
$$\SD(A)= \inf_{L\in \N} \frac{\left|A\cap [1,L]\right|}{L}.$$
(See \cite{SCHNI}; for details on the Schnirelmann density we refer to \cite{NIVEN}.)
Similarly, we can define the Schnirelmann density of a set $A\subseteq \N^k$ by:
$$\SD(A)= \inf_{L\in \N}\frac{\left|A\cap [1,L]^k\right|}{L^k}.$$
Later, we will discuss the Schnirelmann density of some sets of visible points. (We mention that, while for regular density it matters little whether we consider the set of visible points in $\N^k$ or in $\Z^k$, when it comes to Schnirelmann density we will consider only $\N^k$.)
We note that our interest in the Schnirelmann density of sets of visible points started from a question of Moser and Pach, posed in \cite[Problem 64]{MOSER} (which is a forerunner of \cite{PACH}). There they asked about an estimate and bounds for $\SD(V_2)$. They raised a similar question regarding the set of points simultaneously visible from $(1,0)$ and $(0,1)$.

The Schnirelmann density of a set is, in general, smaller than its regular density; as an extreme example, we have $D(\{2,3,4,\ldots\})= 1$, while $\SD\{2,3,4,\ldots\})= 0$. Calculating $|V_2\cap[1,L]^2|$ for some values of $L$, one may be tempted to believe that $\SD(V_2)= D(V_2)$. In fact, not until $L=820$ does one get a square $[1, L]^2$ with $|V_2\cap[1,L]^2|< L^2/\zeta(2)$ (see Sarma~\cite{SARMA}). The following result shows that there are infinitely many counter-examples in every dimension.

\begin{thm}\label{visibility Omega +- theorem}
Let $k \geq 2$, and let $V_k= V(\{(0,\ldots,0)\})$ be the set of points visible from the origin in $\N^k$. Then
\begin{equation}\label{results, omega +-}
\left | V_k \cap [1,L]^k\right| = \frac{L^k}{\zeta(k)} + \Omega_{\pm}(L^{k-1}).
\end{equation}
In particular, the Schnirelmann density of $V_k$ is strictly below the regular density:
\begin{equation}\label{main result, schini}
\SD(V_k)< D(V_k),\qquad k\geq 2.
\end{equation}
\end{thm}

\begin{remark}\label{result, improved omega+}\emph{For $k=2$, it follows from \eqref{intro, montgomery} that the second term on the right-hand side of \eqref{results, omega +-} may be replaced by $\Omega_{\pm}(L \sqrt{\log \log L})$. By Nymann's result~\cite{Nym}, mentioned above, no such improvement is possible for $k\geq 3$. Still, for $k\geq 3$, the $\Omega_{+}$-direction of \eqref{results, omega +-} may be made explicit as follows:
$$\left | V_k \cap [1,L]^k\right| = \frac{L^k}{\zeta(k)} +k\left(\frac{1}{\zeta(k)}- \frac{1}{\zeta(k-1)}\right)L^{k-1} + \Omega_{+}(L^{k-1}).$$
See Remark \ref{improved omega + bound} below for further details.}
\end{remark}

How does one calculate $\SD(V_k)$ for a given $k$? Theorem \ref{visibility Omega +- theorem} guarantees that a finite computation will provide an $L_0$ such that
$$\frac{\left|V_k\cap [1,L_0]^k\right|}{L_0^k}< \frac{1}{\zeta(k)}.$$
Suppose we have an explicit lower bound on $\frac{\left|V_k\cap [1,L]^k\right|}{L^k}- \frac{1}{\zeta(k)}$, which goes to 0 as $L\to \infty$. Then the computation of $\SD(V_k)$ becomes a finite problem. Indeed, we only need to compute $\frac{\left|V_k\cap [1,L]^k\right|}{L^k}$ up to the point where the error becomes smaller in absolute value than $\left|\frac{\left|V_k\cap [1,L_0]^k\right|}{L_0^k}- \frac{1}{\zeta(k)}\right|$.
We will apply this method to compute $\SD(V_2)$ and  $\SD(V_3)$, as follows.

\begin{prop}\label{Schni, proposition for visibiliy for k=2}
For $k=2$,
$$\SD(V_2) = \frac{\left|V_2\cap[1,1276]^2\right|}{1276^2}= 0.60787\ldots< 0.60792\ldots= \frac{1}{\zeta(2)}= D(V_2).$$
\end{prop}

\begin{prop}\label{Schni, proposition for visibiliy for k=3}
For $k=3$,
$$\SD(V_3) = \frac{\left|V_3\cap[1,169170]^3\right|}{169170^3}=0.831907366\ldots<0.831907372\ldots= \frac{1}{\zeta(3)}= D(V_3).$$
\end{prop}
What about $\SD(V_4)$ ? By a computer program, we have verified that 
$\frac{\left|V_4\cap [1,L]^4\right|}{L^4}>\frac{1}{\zeta(4)}$ for every $L\leq 10^9$. P{\'e}termann \cite{PETER-1} opined, based on his computations in \cite{PETER}, on which Lemma \ref{jordan omega + result} below also hinges, that the smallest $L$ for $\frac{\left|V_4\cap [1,L]^4\right|}{L^4}<\frac{1}{\zeta(4)}$ may be of order of magnitude $10^{12}$.

We can similarly deal with the abovementioned problem regarding simultaneous visibility in $\N^2$.
\begin{prop}\label{Schni, proposition visibility from (0,1) and (1,0)}\emph{
Consider the set $A= V(\{(1,0),(0,1)\})$ of
points simultaneously visible from both $(0,1)$ and $(1,0)$ in $\N^2$. Then
$$\SD(A)=  \frac{\left|A\cap[1,7]^2\right|}{7^2}=\frac{15}{49}= 0.306\ldots< 0.322\ldots= \prod_{p\in \mathcal{P}}\left(1-\frac{2}{p^2}\right)= D(A).$$
}
\end{prop}

Can one find algorithmically $\SD(V(S))$ for a given set $S$? It is possible to calculate $|V(S)\cap[1, L]^k|$ for larger and larger values of $L$. If an $L_0$ for which $|V(S)\cap[1, L_0]^k|<L_{0}^k\cdot D(V(S))$ is found, then it is possible in principle to calculate $\SD(V(S))$. To this end, one needs to follow the proof of Theorem \ref{main thm} in an effective way $-$ replace all big oh estimates by estimates with an explicit constant. Let $E= D(V(S))-|V(S)\cap[1, L_0]^k|/L_0^k$.
Once this has been done, we can find an $L_1$ such that 
$$\left|\frac{|V(S)\cap [1,L]^k|}{L^k}- D(V(S))\right|<E,\qquad L\geq L_1.$$
Calculating $|V(S)\cap[1, L]^k|$ for all $L<L_1$, we find $\SD(V(S))$. However, as long as an $L_0$ as above is not found, we cannot apply this method. Moreover, even if we know that $\SD(V(S))<D(V(S))$, since we do not know how to bound $L_0$ from above and thus bound $E$ from below, we do not have an algorithmic way to calculate $\SD(V(S))$.

In the case of Propositions \ref{Schni, proposition for visibiliy for k=2} and \ref{Schni, proposition for visibiliy for k=3}, we know beforehand by Theorem \ref{visibility Omega +- theorem} that the Schnirelmann density is smaller than the regular density.
In their proofs of these propositions, we found $L_0$, bounded $E$ from below, and applied the above method to calculate the Schnirelmann density. For other sets $S$, it might be the case that $\SD(V(S))=~D(V(S))$. In such cases, our methods cannot be used to prove this equality.

A family of finite sets $S\subset \Z^k$, for which we trivially have $\SD(V(S))= D(V(S))$, is the family of sets for which $s(p)= p^k$ for some prime $p$; in this case, $\SD(V(S))= D(V(S))=0$. In view of Theorem \ref{visibility Omega +- theorem}, and Propositions \ref{Schni, proposition for visibiliy for k=2}-\ref{Schni, proposition visibility from (0,1) and (1,0)} above, Examples \ref{example, (0,1),(1,0)}-\ref{example, (0,0,0),(1,0,0)(0,1,0),(0,0,1)} and Tables \ref{table:1},\ref{table:2}, \ref{table:3} below, we raise
\begin{conjecture}\emph{Let $S\subset \Z^k$ be any finite subset. If $D(V(S))>0$, then $\SD(V(S))< D(V(S))$.}
\end{conjecture}
\subsection{Organization of the paper}
In Section \ref{section, preliminaies} we make some preparations towards the proofs. In Sections~\ref{section, selberg} and~\ref{section, selberge, main result} make use of a higher-dimensional Selberg sieve to prove our main result. Section~\ref{occelation from orizin}, we prove Theorem \ref{visibility Omega +- theorem}. Section~\ref{section shini} is devoted to the proofs of Propositions \ref{Schni, proposition for visibiliy for k=2}-\ref{Schni, proposition visibility from (0,1) and (1,0)}. In Section~\ref{section, ergodic} we discuss the visibility problems from an ergodic theoretical viewpoint. Section \ref{section, statistical} contains some numerical results and examples. In Section \ref{Section on visibility in discs} we consider visibility within discs instead of cubes.

\begin{acknowledgement}\emph{
The authors express their gratitude to Y.-F. S. P{\'e}termann and W. Takeda for lengthy correspondence regarding their papers \cite{PETER} and \cite{TAKEDA}. We also thank the referee for the careful comments, which contributed a lot to the readability and organization of the paper.}
\end{acknowledgement}

\section{Preliminaries}\label{section, preliminaies}
Recall that the M\"{o}bius function $\mu$ is defined on the set of positive integers by
\begin{equation}\label{defination of mu}
 \mu(d) = \begin{cases} 
     1, & \qquad d= 1,\\
     (-1)^r, & \qquad  d \ \mbox{is a product of} \ r\  \mbox{distinct primes},\\
     0, &  \qquad d\  \mbox{is not square-free}.
    \end{cases}
     \end{equation}
If $f$ is any multiplicative function, then
\begin{equation}\label{defination of multiplecative}
\sum_{d\mid n}\mu(d)f(d) = \prod_{p\mid n}(1-f(p)), \qquad n\in \N.
\end{equation}
In particular, taking $f \equiv 1$,
\begin{equation}\label{sum of mu}
 \sum_{d\mid n}\mu(d) = \begin{cases} 
     1, & \qquad n= 1,\\
     0, & \qquad n>1.\\
\end{cases}
\end{equation}
If two arithmetic functions $f$ and $g$ (not necessarily multiplicative) are related by
$$f(n) = \sum_{d\mid n}g(d), \qquad n= 1,2,\ldots,$$
then by the M\"{o}bius inversion formula~\cite{Apo}
$$g(n) = \sum_{d\mid n}\mu(d)f\left(n/d\right),\qquad n= 1,2 ,\ldots.$$
Let $D\subset \N$ be a divisor closed set (i.e., if $d\in D$ and $d'|d$, then $d'\in D$).
If
$$f(n) = \sum_{n\mid d:\, d\in D}g(d),$$
then by the dual M\"{o}bius inversion formula~(cf. \cite[Theorem 1.2.3]{MMC}),
$$g(n) = \sum_{n\mid d:\, d\in D}\mu(d/n)f(d)$$
(assuming all series are absolutely convergent).

If $\textbf{x}$ = $(x_1,\ldots,x_k)$ and $\textbf{y}$ = $(y_1,\ldots,y_k)$ are points of $\Z^k$, and $m$ is a positive integer, we write
$\textbf{x} \equiv \textbf{y} \ (\textup{mod} \ m)$
to indicate that $x_i \equiv y_i\ (\mbox{mod}\ m)$ for $i= 1,2,\ldots,k$.
The condition that two distinct points $\textbf{x}, \textbf{y} \in \Z^k$ are mutually visible is equivalent to:
\begin{equation*}
\textbf{x} \not \equiv \textbf{y} \ (\textup{mod} \ p), \qquad p\in \mathcal{P}.
\end{equation*}
If $\textbf{x} \equiv \textbf{y} \ (\textup{mod} \ p)$, then $\textbf{x}$, $ \textbf{y}$ are $p$-\textit{invisible} from each other. (We should exclude the case $\textbf{x}= \textbf{y}$, but this is relevant only to some fixed number of points, and we will ignore it.)
Let $S$ be a finite subset of $\Z^k$.
A point $\textbf{x} \in \Z^k$ is $p$-\textit{invisible} from $S$ if it is $p$-invisible from some point of~$S$. For a square-free integer $d=p_{i_1} p_{i_2}\cdots p_{i_t}$, with $p_{i_j} \in \mathcal{P}$ for $1\leq j \leq t$, a~point $\textbf{x}$ is $d$-\textit{invisible} from $S$ if it is $p_{i_j}$-invisible from $S$ for every $1\leq j\leq t$. (Note that the points of $S$ from which $\textbf{x}$ is $p_{i_j}$-invisible may be distinct for distinct $j$-s.) Let $B$ be as in \eqref{box}. Denote by $I_d$ the set of points $\textbf{x}\in B$, that are $d$-invisible from $S$. Then
\begin{equation*}
    I_d = I_{p_{i_1}}\cap I_{p_{i_2}}\cap\ldots\cap I_{p_{i_t}}.
\end{equation*}
The function $s$ in \eqref{S(p)} is defined on $\mathcal{P}$ only. We extend it to a multiplicative function on the set of square-free integers by 
\begin{equation}
    s(d) = \prod_{p\mid d}s(p), \qquad d\in \N, \ \mu(d) \neq 0,
\end{equation}
(and $s(1)=1$).
\begin{lemma}\label{floor and cielling}
For square-free $d$,
\begin{equation*}\label{bounds of Vd}
  \prod_{i=1}^k \left \lfloor{\frac{L_i}{d}}\right \rfloor s(d)\leq |I_d| \leq  \prod_{i=1}^k\left \lceil{\frac{L_i}{d}} \right \rceil s(d),
\end{equation*}
where $\lfloor ~ \rfloor$ and $\lceil~\rceil$ are the floor and the ceiling functions, respectively.
\end{lemma}
\begin{proof}
For each $p \in \mathcal{P}$, let $S_p$ be a subset of size $s(p)$ of $S$, consisting of points which are mutually non-congruent modulo $p$. Let $d= p_{i_1}p_{i_2}\cdots p_{i_t}$. A point $\textbf{x} \in \Z^k$ is $d$-invisible from $S$ if
\begin{equation}\label{system of equations}
    \textbf{x} \equiv \textbf{x}_j \ (\textup{mod} \ p_{i_j}), \qquad 1\leq j \leq t,
\end{equation}
for some points $\textbf{x}_j \in S_{p_{i_j}}$, $1\leq j \leq t$. By the Chinese Remainder Theorem, for every choice of $\textbf{x}_j$-s, the system of congruences \eqref{system of equations} has a unique solution modulo $d$.
Therefore, every box of the form $[N_1,N_1+d)\times \cdots \times [N_k,N_k+d) \subset B$ for some integers $N_i$, $1\leq i \leq k$, contains exactly one point satisfying \eqref{system of equations}. Hence for each choice of $\textbf{x}_j$-s, the number of points in $B$ satisfying \eqref{system of equations} is between $\prod_{i=1}^k\lfloor{\frac{L_i}{d}}\rfloor$ and $\prod_{i=1}^k\lceil{\frac{L_i}{d}}\rceil$. Now the number of essentially distinct choices of $\textbf{x}_1, \textbf{x}_2, \ldots , \textbf{x}_t$ in \eqref{system of equations} is $s(p_{i_1})s(p_{i_2})\cdots s(p_{i_t}) = s(d)$, and the sets of solutions of congruences are pairwise disjoint for distinct $\textbf{x}_1, \textbf{x}_2, \ldots , \textbf{x}_t$. This proves the lemma.
\end{proof}

\section{Applying Selberg's Sieve Method}\label{section, selberg}
The classical Selberg sieve method is usually used to obtain estimates on the size of certain sets of positive integers, defined by some congruence conditions. Here we use this machinery to deal with our sets of visible points. We note that the idea of using the sieve for higher-dimensional sets was mentioned already by Selberg and used for other problems (see, for example, \cite{Sel, May, Pol, Vat}). 
For the sake of self-cotainedness, we present the development of the tool in this case, following~\cite{Ric, MMC}.

It follows from Lemma \ref{floor and cielling} that
\begin{equation} \label{I_d, expression with big O}
|I_d| =\frac{L_1\cdots L_{k}}{d^k}s(d) + R_d,
\end{equation}
where
\begin{equation}\label{selberg, expresion of Rd}
R_d = O\left(s(d)\cdot\left(\frac{L_1\cdots L_{k-1}}{d^{k-1}} + \frac{L_1\cdots L_{k-2}}{d^{k-2}} + \cdots +  \frac{L_1}{d} + 1\right)\right).
\end{equation}
For a positive real number $z$, set
\begin{equation*}
P(z) = \prod_{p \in \mathcal{P}:\, p< z}p.
\end{equation*}
Denote by
\begin{equation}\label{def of V(S,B,z)}
V(S,B,z)= |\{\textbf{x} \in B: (\textbf{x}-\textbf{a}, P(z))= 1,\,\,  \textbf{a} \in S\}|
\end{equation}
the number of points $\textbf{x} \in B$, such that $\textbf{x}- \textbf{a} \not \equiv \textbf{0} \ (\textup{mod} \ p)$ for all primes $p< z$ and $\textbf{a} \in S$. Here, for $\textbf{b} = (b_1,b_2,\ldots, b_k) \in \Z^k$ and $m\in \N$, we denote $(\textbf{b}, m)= \gcd(b_{1},b_{2},\ldots,b_{k}, m)$.
By inclusion-exclusion
\begin{align}\label{expe of V(S,B,z)}
    V(S,B,z) = \sum_{d\mid P(z)}\mu(d)|I_d| = \sum_{\textbf{x}\in B}\left(\sum_{d\mid P(z):\, \textbf{x}\in I_d} \mu(d)\right).
\end{align}
Let $(\lambda_d)_{d=1}^{\infty}$ be any sequence of real numbers such that $\lambda_1 =1$. We claim that
\begin{equation}\label{expression of V(S,B,z) with lambda squre}
    V(S,B,z) \leq \sum_{\textbf{x}\in B}\left(\sum_{d\mid P(z):\, \textbf{x}\in I_d} \lambda_d\right)^2.
\end{equation}
Indeed, each $\textbf{x} \in B$ contributes 1 to the left-hand side if $(\textbf{x}-\textbf{a}, P(z)) = 1$ for all $\textbf{a} \in S$ and  contributes 0 otherwise. Since $\lambda_1 =1$, the contribution of $\textbf{x}$, with $(\textbf{x}- \textbf{a}, P(z)) = 1$ for all $\textbf{a} \in S$, to the right-hand side is also 1. As the contribution of other points $\textbf{x}$ to the right-hand side is certainly non-negative, this proves \eqref{expression of V(S,B,z) with lambda squre}. 

Interchanging the order of summation in \eqref{expression of V(S,B,z) with lambda squre}, we obtain
\begin{align}\label{V(S,B,z), |I|}
    V(S,B,z) \leq \sum_{\textbf{x}\in B}\left(\sum_{d_1,d_2 \mid P(z):\ \textbf{x}\in I_{[d_1,d_2]}}\lambda_{d_1}\lambda_{d_2} \right)
    = \sum_{d_1,d_2|P(z)}\lambda_{d_1}\lambda_{d_2}|I_{[d_1,d_2]}|,
\end{align}
where $[d_1,d_2]$ denotes the least common multiple of $d_1$ and $d_2$.
By \eqref{I_d, expression with big O} and \eqref{V(S,B,z), |I|}:
\begin{align}\label{V(S,B,z) and error}
\begin{split}
    V(S,B,z) &\leq L_1\cdots L_k\sum_{d_1,d_2\mid P(z)}\lambda_{d_1}\lambda_{d_2}\frac{s([d_1,d_2])}{[d_1,d_2]^{k}} + O\left(\sum_{d_1,d_2\mid P(z)}|\lambda_{d_1}\lambda_{d_2}||R_{[d_1,d_2]}|\right)\\
     &=L_1\cdots L_k\cdot \Sigma_1 + O(\Sigma_2).
\end{split}
\end{align}

Selberg's idea was to choose $\lambda_{d}$ for $d\geq 2$ in such a way that the expression on the right-hand side of $\eqref{V(S,B,z) and error}$ will become as small as possible. To keep $\Sigma_2$ small, we take
\begin{equation}\label{lambda =0}
\lambda_{d} = 0,\  \quad \  d\geq z.
\end{equation}
The remaining $\lambda_{d}$, for $2\leq d < z$ with $d\mid P(z)$, are chosen so as to minimize the quadratic form $\Sigma_1$.
Define a multiplicative function $g$ by
\begin{equation}
    g(d) = \frac{s(d)}{d^{k}\prod_{p\mid d}\left(1- \frac{s(p)}{p^{k}}\right)}, \qquad d\in \N,\ \mu(d)\neq 0.
\end{equation}
Note that 
\begin{equation}\label{p^k/s(p)}
1+ \frac{1}{g(p)} = 1+ \frac{p^k - s(p))}{s(p)}= \frac{p^k}{s(p)}, \qquad  p\in \mathcal{P}.
\end{equation}
Since $s$ is multiplicative, for square-free $d_1$ and $d_2$ we have
\begin{align*}
    \frac{s([d_1,d_2])}{[d_1,d_2]^k} &= \frac{s(d_1)}{d_1^k}\cdot\frac{s(d_2)}{d_2^k}\cdot\frac{(d_1,d_2)^k}{s\left((d_1,d_2)\right)}\\
    &= \frac{s(d_1)}{d_1^k}\cdot\frac{s(d_2)}{d_2^k}\cdot \prod_{p\in \mathcal{P}:\, p\mid (d_1,d_2)}\left(1 + \frac{1}{g(p)}\right).
\end{align*}
Therefore
\begin{align*}
    \Sigma_1 &= \sum_{d_1,d_2<z:\, d_1,d_2\mid P(z)}\lambda_{d_1}\lambda_{d_2} \frac{s(d_1)}{d_1^k}\cdot\frac{s(d_2)}{d_2^k}\cdot\prod_{p\in \mathcal{P}:\, p\mid (d_1,d_2)}\left(1 + \frac{1}{g(p)}\right)\\
    &=\sum_{d_1,d_2<z:\, d_1,d_2\mid P(z)}\lambda_{d_1}\lambda_{d_2} \frac{s(d_1)}{d_1^k}\cdot\frac{s(d_2)}{d_2^k}\sum_{d\mid (d_1,d_2)}\frac{1}{g(d)}\\
    &= \sum_{d<z:\, d \mid P(z)}\frac{1}{g(d)}\sum_{d_1,d_2<z:\, d_1,d_2\mid P(z), d\mid (d_1,d_2)}\lambda_{d_1}\lambda_{d_2} \frac{s(d_1)}{d_1^k}\cdot\frac{s(d_2)}{d_2^k}\\
    &= \sum_{d< z:\, d\mid P(z)}\frac{1}{g(d)}\left(\sum_{l<z:\,l\mid P(z), d\mid l}\lambda_{l} \frac{s(l)}{l^k}\right)^2.
\end{align*}
Thus, under the transformation 
 \begin{equation}\label{expresion of u_d}
     u_d = \sum_{l<z:\, l\mid P(z), d\mid l}\lambda_{l} \frac{s(l)}{l^k},
 \end{equation}
the quadratic form $\Sigma_1$ is reduced to a diagonal form:
\begin{equation}\label{quad form}
    \Sigma_1 = \sum_{d<z:\, d\mid P(z)}\frac{1}{g(d)}u_d^2.
\end{equation}
By the dual M\"{o}bius inversion formula, \eqref{expresion of u_d} yields
\begin{equation}\label{first formula of lambda}
    \lambda_d\frac{s(d)}{d^k} = \sum_{l<z:\, l\mid P(z), d\mid l}\mu(l/d)u_l.
\end{equation}
Since $\lambda_d = 0$ for $d\geq z$ and $\lambda_1 =1$, we obtain
\begin{equation}\label{u_d =0}
u_l = 0, \qquad l\geq z,
\end{equation}
and 
\begin{equation} \label{u_d, lambda_1}
\sum_{l< z:\, l\mid P(z)}\mu(l)u_l = \lambda_1 = 1.
\end{equation}
Put
 \begin{equation}\label{G(z) expresion}
G(z) = \sum_{d< z}\mu^2(d)g(d)= \prod_{p\in \mathcal{P}:\, p<z}(1+ g(p)),
 \end{equation}
and
  \begin{equation} \label{G_k}
G_k(z) = \sum_{\substack{d<z:\, (d,k) =1}}\mu^2(d)g(d) = \prod_{p\in \mathcal{P}:\, p<z, (p,k)=1}(1+ g(p)), \qquad k \in \N.
\end{equation}
By \eqref{quad form}, \eqref{u_d =0} and \eqref{u_d, lambda_1}:
\begin{align*}
\sum_{d<z:\, d \mid P(z)}&\frac{1}{g(d)}\left( u_d - \frac{\mu(d)g(d)}{G(z)}\right)^2\\
&= \sum_{d<z:\, d \mid P(z)}\frac{1}{g(d)}\left( u_d^2 + \mu^2(d) g^2(d)\frac{1}{G^2(z)} - 2\mu(d)u_d g(d)\frac{1}{G(z)}\right)\\
&= \sum_{d<z:\, d \mid P(z)}\frac{1}{g(d)}u_d^2 + \sum_{d<z:\, d \mid P(z)}\mu^2(d)g(d)\frac{1}{G^2(z)}
- 2\sum_{d<z:\, d \mid P(z)}\mu(d)u_d\frac{1}{G(z)}\\
&= \Sigma_1 + \frac{1}{G(z)} - 2\cdot\frac{1}{G(z)}\\
&= \Sigma_1 - \frac{1}{G(z)}.
\end{align*}
Therefore:
\begin{equation}\label{minimal value}
    \Sigma_1 = \sum_{d<z:\, d \mid P(z)}\frac{1}{g(d)}\left( u_d - \frac{\mu(d)g(d)}{G(z)}\right)^2 + \frac{1}{G(z)}.
\end{equation}
Since $g(p)> 0$ for every $p \in \mathcal{P}$, we infer from \eqref{minimal value} that the minimal value of $\Sigma_1$ is $\frac{1}{G(z)}$, and is attained for
\begin{equation}\label{u_d second expression}
    u_d = \frac{\mu(d)g(d)}{G(z)}, \qquad d<z , \mu(d)\neq 0.
\end{equation}
By \eqref{p^k/s(p)}, \eqref{first formula of lambda} and \eqref{u_d second expression}, for square-free $d<z$:
\begin{align*}
\lambda_d &= \frac{d^k}{s(d)}\sum_{l<z:\, l\mid P(z), d\mid l}\mu(l/d)\frac{\mu(l)g(l)}{G(z)}= \frac{d^k}{s(d)}\sum_{t<\frac{z}{d}:\, t\mid P(z), (d,t) =1}\mu(t) \mu(dt)\frac{g(dt)}{G(z)}\\
&= \frac{d^k}{s(d)}\sum_{t<\frac{z}{d}:\, t\mid P(z), (d,t) =1}\mu^2(t) \mu(d)\frac{g(t)g(d)}{G(z)}= \mu(d)g(d)\frac{d^k}{s(d)}\frac{G_d(z/d)}{G(z)}\\
&= \mu(d)g(d)\prod_{p\mid d}\left(1 + \frac{1}{g(p)}\right)\frac{\prod_{p<z/d:\, (p,d)=1}(1+ g(p))}{\prod_{p<z}(1+ g(p))}\\
&= \mu(d)\prod_{p\mid d}(1 + g(p))\frac{\prod_{p<z/d:\, (p,d)=1}(1+ g(p))}{\prod_{p<z}(1+ g(p))}\\
&= \mu(d)\frac{1}{\prod_{z/d \leq p<z:\, (p,d)= 1}(1+ g(p))}.
\end{align*}
Hence $|\lambda_d|\leq |\mu(d)|=1$ for every square-free $d<z$.
It is easy to see that
$$|\{(d_1,d_2): [d_1,d_2]= d\}| = 3^{\omega(d)},\qquad \mu(d)\neq 0,$$
where $\omega(d)$ is the number of prime divisors of $d$.
Therefore, we arrive at
\begin{equation}\label{3^v(d)}
   \Sigma_2\leq \sum_{d_1,d_2\mid P(z)}|R_{[d_1,d_2]}|\leq \sum_{d < z^2:\, d\mid P(z)}3^{\omega(d)}|R_d|.
\end{equation}
By \eqref{selberg, expresion of Rd} and \eqref{3^v(d)}:
\begin{align}\label{final expression of sigma 2}
\begin{split}
\Sigma_2 &\leq O\left(\sum_{d < z^2:\, d\mid P(z)}3^{\omega(d)}s(d)\cdot\left(\frac{L_1\cdots L_{k-1}}{d^{k-1}} + \frac{L_1\cdots L_{k-2}}{d^{k-2}} + \cdots +  \frac{L_1}{d} + 1\right)\right)\\
&= O\left(\sum_{j=1}^{k-1} L_1\cdots L_{j}\cdot\prod_{p<z: p\in \mathcal{P}}\left(1+ \frac{3s(p)}{p^{j}}\right) + \sum_{d<z^2:\, d|P(z)}3^{\omega(d)}s(d) \right).\\
    \end{split}
\end{align}
For $j\geq 2$, the product $\prod_{p<z:p\in \mathcal{P}}\left(1+ \frac{3s(p)}{p^{j}}\right)$ is bounded. For $j=1$:
\begin{equation}\label{selberg, product for j=1}
\begin{split}
    \prod_{p<z: p\in \mathcal{P}}\left(1+ \frac{3s(p)}{p}\right)&\leq \prod_{p<z: p\in \mathcal{P}}\left(1+ \frac{3 r}{p}\right) \leq \prod_{p<z: p\in \mathcal{P}}\exp(3r/p)\\
    &= \exp(3r\sum_{p<z: p \in \mathcal{P}}1/p)= \exp(3r\log \log z + O(1))\\
    &=O(\log^{3r} z),
\end{split}
\end{equation}
where the second last equality follows from \cite[Theorem 4.12]{Apo}. By~\cite{G.R}, the function $\omega$ is bounded as follows:
\begin{equation}\label{hardyram}
    \omega(d) \leq \frac{2\log d}{\log \log d}, \qquad d\geq 3.
\end{equation}
Now for any $\varepsilon>0$ and sufficiently large $d$:
\begin{equation}\label{selberg, (3.r) epsilon expresion}
\begin{split}
 \sum_{d<z^2: d|P(z)}3^{\omega(d)}s(d)&\leq
    \sum_{d<z^2: d|P(z)}(3r)^{\omega(d)}
    \leq \sum_{d<z^2: d|P(z)}(3 r)^{\frac{2\log d}{\log \log d}}\\
    &\leq \sum_{d<z^2: d|P(z)} d^{\varepsilon} \leq z^{2+ \varepsilon}.
\end{split}
\end{equation}
Hence, by \eqref{final expression of sigma 2}, \eqref{selberg, product for j=1}, and \eqref{selberg, (3.r) epsilon expresion}:
\begin{equation}\label{selberg, final expresion of sigma2}
\begin{split}
    \Sigma_2 &=
    O\left(L_1\cdots L_{k-1} + L_1\cdots L_{k-2}+ \cdots + L_1 L_2 + L_1\log^{3r} z + z^{2+ \varepsilon} \right)\\
    &=O\left(L_1\cdots L_{k-1} + L_1\log^{3r} z + z^{2+ \varepsilon} \right).
    \end{split}
\end{equation}

\section{Conclusion of the Proof of Theorem \ref{main thm}}\label{section, selberge, main result}
Denote: 
\begin{equation}\label{G expresion}
G = \sum_{d\in \N}\mu^2(d)g(d).
 \end{equation}
We have
\begin{align*}
G&=\sum_{d\in \N}\mu^2(d)\frac{s(d)}{d^k\prod_{p\in \mathcal{P}: p|d}\left(1-\frac{s(p)}{p^k}\right)}\\
&=\prod_{p\in\mathcal{P}}\left(1+ \frac{s(p)}{p^k-s(p)}\right)=\prod_{p\in\mathcal{P}}\frac{p^k}{p^k-s(p)},
\end{align*}
and therefore
\begin{equation}\label{selberg, expresion of 1/G}
\frac{1}{G} = \prod_{p\in \mathcal{P}}\left(1-\frac{s(p)}{p^k}\right).
 \end{equation}
For any $\varepsilon>0$:
\begin{equation}\label{selberg, expresion of G(z)-G in term of epsilon}
\begin{split}
\frac{1}{G(z)}- \frac{1}{G} &= 
\frac{1}{\prod_{p\in \mathcal{P}:\, p<z}(1+ g(p))}- \frac{1}{\prod_{p\in \mathcal{P}}(1+ g(p))}\\
&=\frac{1}{\prod_{p\in \mathcal{P}}\left(1+ g(p)\right)}\left(\prod_{p\in \mathcal{P}:\, p\geq z}(1+ g(p)) - 1\right)\\
&= O\left(\prod_{p\in \mathcal{P}:\, p\geq z}(1+ g(p))\right)
= O\left(\sum_{d\geq z}\mu^2(d)g(d)\right)= O\left(\sum_{d\geq z}g(d)\right)\\
&=  O\left(\sum_{d\geq z}\frac{s(d)}{d^k\prod_{p|d}\left(1-\frac{s(p)}{p^k}\right)}\right)=  O\left(\sum_{d\geq z}\frac{s(d)}{d^k\prod_{p\in\mathcal{P}}\left(1-\frac{s(p)}{p^k}\right)}\right)\\
&=O\left(\sum_{d\geq z}\frac{s(d)}{d^k}\right)= O\left(\sum_{d\geq z}\frac{r^{\omega(d)}}{d^k}\right)\\
&=O\left(\sum_{d\geq z}\frac{d^{\varepsilon}}{d^k}\right)= O\left(\frac{1}{z^{k-1-\varepsilon}}\right).
\end{split}
\end{equation}
By \eqref{V(S,B,z) and error}, \eqref{minimal value}, \eqref{u_d second expression}, \eqref{selberg, final expresion of sigma2}, \eqref{selberg, expresion of 1/G}, and \eqref{selberg, expresion of G(z)-G in term of epsilon}:
\begin{equation}\label{main result, V(S,B,z) error term for k=2, k>2}
\begin{split}
V(S,B,z) &\leq \frac{L_1\cdots L_k}{G(z)} + O\left(L_1\cdots L_{k-1}+ L_1\log^{3r} z + z^{2+ \varepsilon}\right)\\
&=L_1\cdots L_k\cdot\prod_{p\in \mathcal{P}}\left(1 - \frac{s(p)}{p^k}\right) + L_1\cdots L_k \cdot O\left(\frac{1}{z^{k-1- \varepsilon}} \right)\\
&~~~+  O\left(L_1\cdots L_{k-1} + L_1\log^{3r} z + z^{2 + \varepsilon}\right)\\
&=     L_1\cdots L_k\cdot\prod_{p\in \mathcal{P}}\left(1 - \frac{s(p)}{p^k}\right)\\
&~~~+ \begin{cases}O\left(L_1 L_2 \cdot  \frac{1}{z^{1- \varepsilon}} + L_1\log^{3r} z + z^{2+ \varepsilon}\right), & k=2,\\
O\left(L_1\cdots L_k \cdot \frac{1}{z^{k- 1-\varepsilon}}  + L_1\cdots L_{k-1}+ L_1\log^{3r} z+ z^{2 +\varepsilon}\right), & k>2.\end{cases}
\end{split}
\end{equation}

For $k=2$, the candidates for the optimal choice of $z$ are those values of $z$ which make two of the three terms $E_{2,1}= \frac{L_1L_2}{z^{1-\varepsilon}}$, $E_{2,2}= L_1\log^{3r}z$, and $E_{2,3}= z^{2+\varepsilon}$ equal (up to a big oh factor).
The possibility $E_{2,1}= E_{2,2}$ implies $ z\simeq L_2^{1+\varepsilon}$, the error in \eqref{main result, V(S,B,z) error term for k=2, k>2} being $O( L_1\log^{3r}L_2 + L_2^{2+\varepsilon})$.
The possibility $E_{2,1}= E_{2,3}$ implies $z\simeq(L_1L_2)^{\frac{1}{3}}$, and the error in \eqref{main result, V(S,B,z) error term for k=2, k>2} is $O\left( L_1\log^{3r}(L_1) + (L_1L_2)^{\frac{2}{3}+ \varepsilon}\right)$.
The possibility $E_{2,2}= E_{2,3}$ implies $z\simeq L_1^{\frac{1}{2}-\varepsilon}$, and the error in \eqref{main result, V(S,B,z) error term for k=2, k>2} is $O\left( L_1^{\frac{1}{2}+ \varepsilon}L_2 + L_1\log^{3r}L_1\right)$.
A routine calculation shows that the error is in any case $O\left(\max\left\{(L_{1}L_2)^{2/3+ \varepsilon}, L_1\log^{3r}L_2\right\}\right)$.

For $k=3$ we readily verify that, by choosing $z= L_k^{3/(2(k-1))}$, we make all addends in the error term dominated by $L_1\cdots L_{k-1}$.

Altogether, 
\begin{align*}
V(S,B) &\leq 
L_1\cdots L_{k}\cdot\prod_{p}\left(1 - \frac{s(p)}{p^k}\right) + \begin{cases} O\left(\max\left\{(L_{1}L_2)^{2/3+ \varepsilon}, L_1\log^{3r}L_2\right\}\right), & k=2,\\
O(L_1\cdots L_{k-1}), & k>2,
\end{cases}
\end{align*}
which completes the proof of Theorem \ref{main thm}.

\section{Proof of Theorem \ref{visibility Omega +- theorem}}\label{occelation from orizin}
Recall that, for a positive integer $k$, the \textit{Jordan totient function} $J_k$ is defined by:
$$J_k(n)= n^k\prod_{p\in \mathcal{P}: p|n}\left(1-\frac{1}{p^k}\right), \qquad n\geq 1.$$
An alternative expression for $J_k(n)$ is:
\begin{equation}\label{jordan, formula for jordan function in terms of mobious fuction}
    J_k(n)= \sum_{d|n}\mu(d)\left(\frac{n}{d}\right)^k, \qquad n\geq 1.
\end{equation}
We agree that
$$J_0(n)= \begin{cases} 1,& \qquad n=1,\\
0, &\qquad  n\geq 2.
\end{cases}$$
Note that $J_1$ is Euler's totient function $\varphi$. (For more on $J_k$, we refer, for example, to~\cite{COHEN}.)
For $k\geq 1$ and $m\geq 1$, denote:
\begin{equation}\label{jordan, expresion of E(kim)}
    E_{k,i,m}= V_k \cap \left\{(x_1,\ldots,x_k)| \, x_i=m=\max\{ x_1,\ldots, x_k\}\right\},\qquad 1\leq i\leq k.
\end{equation}
A point $(x_1,\ldots, x_k)$ with $x_i=m$ is visible if and only if some prime divisor of $m$ divides also $x_1,\ldots, x_{i-1}, x_{i+1},\ldots, x_k$. Let $Q=\{q_1,\ldots,q_e\}$ be the set of prime divisors of $m$.
For a set $R\subseteq Q$, denote by $I(R)$ the set of points $(x_1,\ldots, x_{i-1},m, x_{i+1},\ldots, x_k)\in [1,m]^k$ such that $q\mid x_j$ for every $q\in R$ and every $j$. Clearly, 
$$|I(R)|= m^{k-1}\prod_{q\in R}\frac{1}{q^{k-1}}.$$
By inclusion-exclusion and \eqref{jordan, expresion of E(kim)}
$$E_{k,i,m}= \sum_{R\subseteq Q}(-1)^{|R|}|I_R|= m^{k-1}\sum_{R\subseteq Q}(-1)^{|R|}\prod_{q\in R}\frac{1}{q^{k-1}}= J_{k-1}(m),\qquad 1\leq i\leq k,\,  m\geq 1.$$
This clearly implies that
$$\left|\bigcap_{j=1}^r E_{k,i_{j},m}\right|= J_{k-r}(m), \qquad 1\leq i_{1}<i_{2}<\ldots< i_{r}\leq k.$$
We have
$$ V_k\cap [1,L]^k = \bigcup_{m=1}^{L}\bigcup_{i=1}^{k}E_{k,i,m},$$
where the external union is disjoint.
By the inclusion-exclusion principle

\begin{equation}\label{jordan visible lattice point sum}
\begin{split}
    \left| V_k\cap [1,L]^k\right|
    &= \sum_{m=1}^{L}\Bigg(\sum_{i=1}^{k}\left|E_{k,i,m}\right|- \sum_{1\leq i_{1}< i_{2}\leq k} \left|\bigcap_{j=1}^2 E_{k,i_{j},m}\right|\\
     &~~~~~~~~~~~~~~~~~~~~~~~~+\cdots +(-1)^{k-1}\left|\bigcap_{i=1}^k E_{k,i,m}\right|\Bigg)\\
    &=\sum_{m=1}^L\left(kJ_{k-1}(m) -\binom{k}{2}J_{k-2}(m) + \cdots + (-1)^{k-1}\binom{k}{k}J_0(m) \right).
    \end{split}
\end{equation}

\begin{lemma}\label{jordan, sum of J(k-1) with erroor}
For $k\geq 3$,
$$k\sum_{m=1}^LJ_{k-1}(m)= \frac{L^k}{\zeta(k)} +E_k(L),$$
where
$$E_k(L)= \frac{k}{2}\cdot\frac{L^{k-1}}{\zeta(k-1)}- k\sum_{d=1}^L\mu(d)\left(\frac{L}{d}\right)^{k-1}\left\{\frac{L}{d}\right\} + O(L^{k-2}\log L).$$
\end{lemma}
\begin{remark}
\emph{In \cite[Lemma 3.3]{SUKUMAR} and \cite[Lemma 4.2]{TAKEDA}, very similar expressions have been obtained, and we could have used their calculations to shorten our proof. However, for self-containedness, we provide a full proof.}
\end{remark}
\begin{remark}
\emph{Erd\H{o}s and Shapiro's~\cite{ERDOS} result, mentioned in the introduction, is that the error $E_k(L)$ in the lemma changes sign infinitely often for $k=2$. However, Adhikari \cite[Theorem 3]{SUKUMAR} showed that $E_k(L)$ is eventually positive for $k\geq 3$.}    
\end{remark}
\begin{proof}[Proof of Lemma \ref{jordan, sum of J(k-1) with erroor}] 
By \eqref{jordan, formula for jordan function in terms of mobious fuction}:
\begin{equation}\label{jordan ksum(k-1)}
 \begin{split}
    k\sum_{m=1}^LJ_{k-1}(m) =
        k\sum_{m=1 }^L\sum_{d|m}\mu(d)\left(\frac{m}{d}\right)^{k-1}
        =k\sum_{d=1}^L\mu(d)\sum_{q=1}^{\left\lfloor L/d\right\rfloor} q^{k-1}.
    \end{split}
\end{equation}
By \cite{SURY},
\begin{equation}\label{jordan bernoli relation}
\sum_{q=1}^{\lfloor L/d\rfloor}q^{k-1} = \frac{1}{k}\sum_{j=0}^{k-1}\binom{k}{j}B_j\left\lfloor\frac{L}{d}\right\rfloor^{k-j},
\end{equation}
where $B_0(=1), B_1 (= 1/2), B_2, \ldots$, are the Bernoulli numbers.
From \eqref{jordan ksum(k-1)} and \eqref{jordan bernoli relation},
\begin{equation*}
    k\sum_{m=1}^LJ_{k-1}(m)= \sum_{d=1 }^L\mu(d)\sum_{j=0}^{k-1}\binom{k}{j}B_j\left\lfloor\frac{L}{d}\right\rfloor^{k-j}.
\end{equation*}
We split the inner sum on the right-hand side into three parts:
\begin{equation}\label{jordan ksum(k-1) with bernouli}
\begin{split}
    k\sum_{m=1}^LJ_{k-1}(m) &= \sum_{d=1}^L\mu(d)\left\lfloor\frac{L}{d}\right\rfloor^k + \frac{k}{2}\sum_{d=1}^L\mu(d)\left\lfloor\frac{L}{d}\right\rfloor^{k-1}\\
    &~~~+\sum_{d=1}^L\mu(d)\sum_{j=2}^{k-1}\binom{k}{j}B_j\left\lfloor\frac{L}{d}\right\rfloor^{k-j}\\
    &= S_1 + S_2 + S_3.
\end{split}
\end{equation}
First, we deal with $S_3$. Take a constant $C_k$ for which:
$$\left|\sum_{j=2}^{k-1}\binom{k}{j}B_j\left\lfloor\frac{L}{d}\right\rfloor^{k-j}\right|\leq C_k\left(\frac{L}{d}\right)^{k-2},\qquad d\leq L.$$
We have: 
\begin{equation}\label{jordan |S3|}
\begin{split}
|S_3|&\leq \sum_{d=1}^L\left|\sum_{j=2}^{k-1}\binom{k}{j}B_j \left\lfloor\frac{L}{d}\right\rfloor^{k-j}\right|
\leq C_k\sum_{d=1}^L\left(\frac{L}{d}\right)^{k-2}
= O(L^{k-2}\log L).
\end{split} 
\end{equation}
(The logarithmic factor is actually required only for $k=3$.) Now rewrite $S_1$ in the form:
\begin{equation}\label{Jordan s11,s12,s13}
\begin{split}
    S_1 &= \sum_{d=1}^L\mu(d)\left\lfloor\frac{L}{d}\right\rfloor^k
    = \sum_{d=1}^L\mu(d)\left(\frac{L}{d}- \left\{\frac{L}{d}\right\}\right)^k\\
    &= \sum_{d=1}^L\mu(d)\sum_{j=0}^{k}(-1)^j\binom{k}{j}\left(\frac{L}{d}\right)^{k-j}\left\{\frac{L}{d}\right\}^j\\
    &= \sum_{d=1}^L\mu(d)\left(\frac{L}{d}\right)^k - k\sum_{d=1}^L\mu(d)\left(\frac{L}{d}\right)^{k-1}\left\{\frac{L}{d}\right\}\\
    &~~~+ \sum_{d=1}^L\mu(d)\sum_{j=2}^{k}(-1)^j\binom{k}{j}\left(\frac{L}{d}\right)^{k-j}\left\{\frac{L}{d}\right\}^j\\
    &= S_{11}+ S_{12} + S_{13}.
\end{split}
\end{equation}
Similarly to \eqref{jordan |S3|}, we show that
\begin{equation}\label{jordan |S13|}
\begin{split}
    S_{13} =O(L^{k-2}\log L).
\end{split}
\end{equation}
For $S_{11}$:
\begin{equation}\label{jordan |s11|}
\begin{split}
    S_{11} &= \sum_{d=1}^{\infty}\mu(d)\left(\frac{L}{d}\right)^k- \sum_{d= L+1}^{\infty}\mu(d)\left(\frac{L}{d}\right)^k\\
    &= \frac{L^k}{\zeta(k)} +O\left(L^k\int_{L}^{\infty}\frac{1}{x^k}dx\right)\\
    &=\frac{L^k}{\zeta(k)} + O(L).
    \end{split}
\end{equation}
By \eqref{Jordan s11,s12,s13}, \eqref{jordan |S13|} and \eqref{jordan |s11|}:
\begin{equation}\label{jordan S1}
    S_1 = \frac{L^k}{\zeta(k)} -k\sum_{d=1}^L\mu(d)\left(\frac{L}{d}\right)^{k-1}\left\{\frac{L}{d}\right\}+ O(L^{k-2}\log L)
\end{equation}
Now $S_2$ is (except for the additional $k/2$ factor) the same as $S_1$, with $k$ replaced by $k-1$. Hence the same calculations yield:
\begin{equation}\label{jordan S2}
\begin{split}
    S_2&= \frac{k}{2}\cdot\frac{L^{k-1}}{\zeta(k-1)} -\frac{k(k-1)}{2}\sum_{d=1}^L\mu(d)\left(\frac{L}{d}\right)^{k-2}\left\{\frac{L}{d}\right\}+ O(L^{k-3}\log L)\\
    &= \frac{k}{2}\frac{L^{k-1}}{\zeta(k-1)} + O(L^{k-2}\log L).
\end{split}
\end{equation}
From \eqref{jordan ksum(k-1) with bernouli}, \eqref{jordan |S3|}, \eqref{jordan S1} and \eqref{jordan S2} we obtain
\begin{equation}\label{jordan k.sum with s1 s2 s3}
\begin{split}
    k\sum_{m=1}^L J_{k-1}(m) &= \frac{L^k}{\zeta(k)} + \frac{k}{2}\frac{L^{k-1}}{\zeta(k-1)}- k\sum_{d=1}^L\mu(d)\left(\frac{L}{d}\right)^{k-1}\left\{\frac{L}{d}\right\} + O(L^{k-2}\log L),
\end{split}
\end{equation}
as required.
\end{proof}
\begin{lemma}\emph{\cite[Lemma 4.1]{TAKEDA}}\label{jordan omega - result}
For every $k\geq 3$, there exists an $\varepsilon =\varepsilon(k)>0$, such that
$$M(L)=\sum_{d=1}^L\frac{\mu(d)}{d^{k-1}}\left\{\frac{L}{d}\right\}<-\varepsilon$$
for infinitely many positive integers $L$.
\end{lemma}

\begin{lemma}\emph{\cite[p. 318]{PETER}} \label{jordan omega + result}
For every $k\geq 3$, there exists an $\varepsilon=\varepsilon(k)>0$, such that
\begin{equation}\label{jordan, petermann equation (29)}
m(L)=\sum_{d=1}^\infty\frac{\mu(d)}{d^{k-1}}\left\{\frac{L}{d}\right\}>\varepsilon
\end{equation}
for infinitely many positive integers $L$.
\end{lemma}
Note that Lemma \ref{jordan omega + result} holds also if we take the sum in \eqref{jordan, petermann equation (29)} up to $L$ instead of $\infty$, namely replace $m(L)$ by $M(L)$. In fact, this follows readily from:
\begin{equation}\label{absolute value of m(L)- M(L)}
    |m(L)- M(L)|= \left|\sum_{d=L+1}^\infty\frac{\mu(d)}{d^{k-1}}\left\{\frac{L}{d}\right\}\right|\leq \sum_{d=L+1}^{\infty}\frac{1}{d^{k-1}}\xrightarrow[L\to \infty]{} 0.
\end{equation}
Therefore
\begin{equation}\label{equality of limsup and liminf of M(L0 and  m(L)}
    \liminf\limits_{L\rightarrow \infty}m(L)= \liminf\limits_{L\rightarrow \infty}M(L)\quad\mbox{and}\quad  \limsup\limits_{L\rightarrow \infty}m(L)= \limsup\limits_{L\rightarrow \infty}M(L). 
\end{equation}
It follows from \eqref{jordan visible lattice point sum} and Lemma \ref{jordan, sum of J(k-1) with erroor} that:
\begin{equation}\label{jordan, number of visible points with fractional part in error}
    \begin{split}
        \left|V_k \cap [1,L]^k\right|&= \frac{L^k}{\zeta(k)} + E_k(L) -\binom{k}{2}\cdot \frac{1}{k-1}\left(\frac{L^{k-1}}{\zeta(k-1)} + E_{k-1}(L)\right)\\
        &~~~+ \ldots + (-1)^{k-1}\binom{k}{k}\\
        &= \frac{L^k}{\zeta(k)} + \frac{k}{2}\cdot\frac{L^{k-1}}{\zeta(k-1)} - k\sum_{d=1}^L\mu(d)\left(\frac{L}{d}\right)^{k-1}\left\{\frac{L}{d}\right\}\\
        &~~~- \binom{k}{2}\frac{1}{k-1}\cdot\frac{L^{k-1}}{\zeta(k-1)} + O(L^{k-2}\log L)\\
        &= \frac{L^k}{\zeta(k)}- kM(L)L^{k-1}+ O(L^{k-2}\log L).
    \end{split}
\end{equation}
By Lemma \ref{jordan omega - result} and Lemma \ref{jordan omega + result},
$$\left | V_k \cap [1,L]^k\right| = \frac{L^k}{\zeta(k)} + \Omega_{\pm}(L^{k-1}),$$
which completes the proof of Theorem \ref{visibility Omega +- theorem}.
\begin{remark}\label{improved omega + bound}\emph{
We have claimed in Remark \ref{result, improved omega+} that
$$\left | V_k \cap [1,L]^k\right| = \frac{L^k}{\zeta(k)} +k\left(\frac{1}{\zeta(k)}- \frac{1}{\zeta(k-1)}\right)L^{k-1} + \Omega_{+}(L^{k-1}).$$
Indeed, denote 
\begin{equation}\label{expression of h_k(L)}
h_{k}(L)= \sum_{d=1}^\infty\frac{\mu(d)}{d^{k-1}}\left(\frac{1}{2}- \left\{\frac{L}{d}\right\}\right), \qquad k\geq 3.
\end{equation}
By \cite[p. 311]{PETER} (see also \cite{SUKUMAR}),
\begin{equation}\label{liminf expression from petermann p. 1}
\liminf\limits_{L\rightarrow \infty}\frac{E_{k}}{kL^{k-1}}= \liminf\limits_{L\rightarrow \infty}h_{k}(L)- \frac{1}{\zeta(k)}.
\end{equation}
By \eqref{equality of limsup and liminf of M(L0 and  m(L)} and \eqref{expression of h_k(L)}
\begin{equation}\label{expression of h(L) in form of lim sup}
\begin{split}
    \liminf\limits_{L\rightarrow \infty}h_{k}(L)&= \frac{1}{2\zeta(k-1)}- \limsup\limits_{L\rightarrow \infty}m(L)\\
    &=\frac{1}{2\zeta(k-1)}- \limsup\limits_{L\rightarrow \infty}M(L),
\end{split}
\end{equation}
which using \eqref{liminf expression from petermann p. 1}, implies that
\begin{equation}\label{liminf of E(k) expression with zeta}
\liminf\limits_{L\rightarrow \infty}\frac{E_{k}}{kL^{k-1}}= \frac{1}{2\zeta(k-1)}- \limsup\limits_{L\rightarrow \infty}M(L) - \frac{1}{\zeta(k)}.
\end{equation}
By Lemma \ref{jordan, sum of J(k-1) with erroor},
\begin{equation}\label{lim sup expression from lemma 2}
\limsup\limits_{L\rightarrow \infty}\frac{E_{k}}{kL^{k-1}}= \frac{1}{2\zeta(k-1)}- \liminf\limits_{L\rightarrow \infty}M(L).
\end{equation}
By \cite[Theorem 2]{PETER},
\begin{equation}\label{limsum m(L) and liminf m(L)}
\begin{split}
 \liminf\limits_{L\rightarrow \infty}\frac{E_{k}}{kL^{k-1}}&=-\limsup\limits_{L\rightarrow \infty}\frac{E_{k}}{kL^{k-1}}.
 \end{split}
\end{equation}
Therefore,
\begin{equation*}
\begin{split}
  \frac{1}{2\zeta(k-1)}- \limsup\limits_{L\rightarrow \infty}M(L)- \frac{1}{\zeta(k)}&= -\frac{1}{2\zeta(k-1)} + \liminf\limits_{L\rightarrow \infty}M(L),
 \end{split}
\end{equation*}
so that
\begin{equation}\label{limsum m(L) and liminf m(L) with zeta(k-1) -  zeta(k)}
\begin{split}
    \liminf\limits_{L\rightarrow \infty}M(L)&=  \frac{1}{\zeta(k-1)}- \frac{1}{\zeta(k)}- \limsup\limits_{L\rightarrow \infty}M(L).
 \end{split}
\end{equation}
From Lemma \ref{jordan omega + result}, \eqref{equality of limsup and liminf of M(L0 and  m(L)} and \eqref{limsum m(L) and liminf m(L) with zeta(k-1) -  zeta(k)}, it follows that there exist an $\varepsilon>0$ such that
$M(L)< \frac{1}{\zeta(k-1)}- \frac{1}{\zeta(k)}- \varepsilon$ for infinitely many $L$.
Hence our claim follows from \eqref{jordan, number of visible points with fractional part in error}.
}
\end{remark}

\section{Schnirelmann Densities}\label{section shini}
\begin{proof}[{Proof of Proposition \ref{Schni, proposition for visibiliy for k=2}}]
As explained before the statement of the proposition, 
the main thing we need is an effective lower bound on the difference $\frac{\left|V_2\cap[1,L]^2\right|}{L^2}-\frac{1}{\zeta(2)}$. One can probably derive such a bound from any of the papers dealing with $D(V_2)$, mentioned in Section \ref{sec, intro}. Here we show it directly, similarly to the calculations in the proof of Lemma \ref{jordan, sum of J(k-1) with erroor}.
For an arbitrary fixed $L$, denote by $\theta_d$ the fractional part of $L/d$.
From \eqref{jordan visible lattice point sum}:
\begin{equation}\label{Schni, number of visible point for k=2}
\begin{split}
\left|V_2\cap [1,L]^2\right|&= \sum_{n=1}^L2J_1(n)- \sum_{n=1}^LJ_0(n)
= 2\sum_{n=1}^L\sum_{d|n}\mu(d)\cdot\frac{n}{d}-1\\
&=2\sum_{d=1}^L\mu(d)\left(1+2+\cdots + \left\lfloor\frac{L}{d}\right\rfloor\right)-1\\
&= \sum_{d=1}^L\mu(d)\left\lfloor\frac{L}{d}\right\rfloor\left(\left\lfloor\frac{L}{d}\right\rfloor+1\right)-1\\
&= \sum_{d=1}^L\mu(d)\left(\frac{L}{d}-\theta_d\right)\left(\frac{L}{d}+1-\theta_d\right)-1\\
&= \sum_{d=1}^L\mu(d)\left(\left(\frac{L}{d}\right)^2 + (1-2\theta_d)\frac{L}{d}- \theta_d(1-\theta_d)\right)-1.
\end{split}
\end{equation}
Hence:
\begin{equation}\label{Schni, density of visible point for k=2}
\begin{split}
\frac{\left|V_2\cap [1,L]^2\right|}{L^2}&=\sum_{d=1}^L\frac{\mu(d)}{d^2}+ \frac{1}{L}\sum_{d=1}^L\frac{\mu(d)}{d}(1-2\theta_d) - \frac{1}{L^2}\sum_{d=1}^L\mu(d)\theta_d(1-\theta_d)-\frac{1}{L^2}\\
&=\sum_{d=1}^\infty\frac{\mu(d)}{d^2}-\sum_{d=L+1}^\infty\frac{\mu(d)}{d^2} + \frac{1}{L}\sum_{d=1}^L\frac{\mu(d)}{d}(1-2\theta_d)\\
&~~~- \frac{1}{L^2}\sum_{d=1}^L\mu(d)\theta_d(1-\theta_d)-\frac{1}{L^2}\\
&\geq \frac{6}{\pi^2}-\sum_{d=L+1}^\infty\frac{1}{d(d-1)}+\frac{1}{L}-\frac{1}{L}\sum_{d=2}^L\frac{1}{d}- \frac{1}{L^2}\sum_{d=1}^L\frac{1}{4}-\frac{1}{L^2}\\
&= \frac{6}{\pi^2}-\frac{1}{L}\sum_{d=2}^L\frac{1}{d}-\frac{1}{4L}-\frac{1}{L^2}\geq \frac{6}{\pi^2} - \frac{\log L}{L}, \qquad L\geq 9.
\end{split}
\end{equation}
As mentioned in Section \ref{sec, intro}, for $L_0=820$, the error is actually negative (and $L_0$ is the smallest number with this property). The error there is:
$$E= \frac{\left|V_2\cap [1,820]^2\right|}{820^2}- \frac{6}{\pi^2}= -0.000028\ldots.$$
The bound $\frac{\log L}{L}$ on the error decreases as $L$ increases. For $L= 5\cdot10^5$ we have 
$$\frac{\log L}{L}= 0.000026\ldots< |E|.$$
A simple program, run on Mathematica, yields:
$$\SD(V_2) = \min_{820\leq L\leq 5\cdot 10^5}\frac{\left|V_2\cap [1,L]^2\right|}{L^2}= \frac{\left|V_2\cap [1,1276]^2\right|}{1276^2}= 0.607877\ldots.$$
\end{proof}
\begin{proof}[{Proof of Proposition \ref{Schni, proposition for visibiliy for k=3}}]
For arbitrary fixed $L$,
\begin{equation}\label{schni, number of visible points for k=3}
\begin{split}
\left|V_3\cap[1,L]^3\right|&= \sum_{m=1}^L3J_2(m)- \sum_{m=1}^L3J_1(m) + \sum_{m=1}^LJ_0(m)\\
    &= 3\sum_{d=1}^L\mu(d)\left(1^2 + 2^2 +\cdots +\left\lfloor\frac{L}{d}\right\rfloor^2\right)\\
    &~~~-3\sum_{d=1}^L\mu(d)\left(1 + 2 +\cdots +\left\lfloor\frac{L}{d}\right\rfloor\right) +1\\
    &=\frac{1}{2}\cdot\sum_{d=1}^L\mu(d)\Bigg(\left\lfloor\frac{L}{d}\right\rfloor\left(\left\lfloor\frac{L}{d}\right\rfloor+1\right)\left(2\left\lfloor\frac{L}{d}\right\rfloor+1\right)\\
    &~~~~~~~~~~~~~~~~~~~~~~~~~~~~~~~~~~~~~-3\left\lfloor\frac{L}{d}\right\rfloor\left(\left\lfloor\frac{L}{d}\right\rfloor+1\right)\Bigg)+1\\
    &=\sum_{d=1}^L\mu(d)\left(\left(\frac{L}{d}-\theta_d\right)^3-\left(\frac{L}{d}-\theta_d\right)\right)+1.
\end{split}
\end{equation}
Hence:
\begin{equation}\label{Schni, density of visible point for k=3}
\begin{split}
\frac{\left|V_3\cap[1,L]^3\right|}{L^3}&= \sum_{d=1}^L\frac{\mu(d)}{d^3}- \frac{3}{L}\cdot\sum_{d=1}^L\frac{\mu(d)}{d^2}\theta_d
+ \frac{1}{L^2}\cdot\sum_{d=1}^L\frac{\mu(d)}{d}(3\theta_d^2-1)\\
&~~~+ \frac{1}{L^3}\cdot\sum_{d=1}^L\mu(d)(\theta_d-\theta_d^3)+ \frac{1}{L^3}\\
&\geq \frac{1}{\zeta(3)}- \sum_{d= L+1}^\infty\frac{\mu(d)}{d^3}-\frac{3}{L}\cdot\sum_{d=1}^L\frac{\mu(d)}{d^2}\theta_d- \frac{1}{L^2}\sum_{d=1}^L\frac{2}{d}- \frac{1}{L^3}\sum_{d=1}^L1\\
&\geq \frac{1}{\zeta(3)} - \frac{1}{2L^2}- \frac{3\zeta(2)}{L}- \frac{2\log L +2}{L^2}- \frac{1}{L^2}\\
&\geq \frac{1}{\zeta(3)}- \frac{3\zeta(2)}{L}- \frac{2\log L +4}{L^2}.
\end{split}
\end{equation}
By computer we find that the smallest $L$, for which the error is negative, is $L_0= 122760$.
The error there is:
$$ E=\frac{\left|V_3\cap[1,122760]^3\right|}{122760^3}- \frac{1}{\zeta(3)}= -2.95313\times10^{-9}. $$
For $L\geq 10^{10}$, the bound on the error in \eqref{Schni, density of visible point for k=3} is less than $|E|$.
Hence:
$$\SD(V_2) = \min_{122760\leq L\leq 10^{10}}\frac{\left|V_2\cap [1,L]^2\right|}{L^3}= \frac{\left|V_2\cap [1,169170]^3\right|}{169170^3}= 0.831907366\ldots.$$
\end{proof}
\begin{proof}[Proof of Proposition \ref{Schni, proposition visibility from (0,1) and (1,0)}]
Let $L$ be an arbitrary fixed positive integer. For any prime~$p$, denote by $V_p'= [1,L]^2- I_p$ the set of points of $[1,L]^2$, which are $p$-visible from both $(1,0)$ and $(0,1)$. Let $s$ be an arbitrary positive integer, to be determined later. Put $P_s= p_1\cdots p_s$.
We have 
$$A\cap[1,L]^2=\bigcap_{p\in \mathcal{P}}V_p' =\bigcap_{i=1}^sV_{p_i}'- \bigcup_{i=s+1}^{\pi(L)}I_{p_i},$$
where $\pi(L)$ is the number of primes not exceeding $L$. Therefore:
\begin{equation}\label{two points visibility, proof first expresion}
|A\cap[1,L]^2|\geq \left|\bigcap_{i=1}^sV_{p_i}'\right|- \left|\bigcup_{i=s+1}^{\pi(L)}I_{p_i}\right|\geq \sum_{d|P_s}\mu(d)|I_d|- \sum_{i=s+1}^{\pi(L)}|I_{p_i}|.
\end{equation}
By Lemma \ref{floor and cielling}, for appropriate numbers $\theta_d\in [0,1)$ we have
\begin{equation}\label{two points, expresion of A}
\begin{split}
  \sum_{d|P_s}\mu(d)|I_d|   & =\sum_{d|P_s:\,\mu(d)=1}|I_d|- \sum_{d|P_s:\,\mu(d)=-1}|I_d|\\
    &\geq \sum_{d|P_s:\,\mu(d)=1}s(d)\left\lfloor\frac{L}{d}\right\rfloor^2- \sum_{d|P_s:\,\mu(d)=-1}s(d)\left\lceil\frac{L}{d}\right\rceil^2\\
    &= \sum_{d|P_s:\,\mu(d)=1}s(d)\left(\frac{L}{d}-\theta_d\right)^2 - \sum_{d|P_s:\,\mu(d)=-1}s(d)\left(\frac{L}{d}+\theta_d\right)^2\\
    &= L^2\sum_{d|P_s:\,\mu(d)=1}s(d)\frac{1}{d^2}- 2L \sum_{d|P_s:\,\mu(d)=1}\frac{s(d)}{d}\theta_d + \sum_{d|P_s:\,\mu(d)=1}s(d)\theta_d^2\\
    &~~~-L^2\sum_{d|P_s:\,\mu(d)=-1}s(d)\frac{1}{d^2}- 2L \sum_{d|P_s:\,\mu(d)=-1}\frac{s(d)}{d}\theta_d - \sum_{d|P_s:\,\mu(d)=-1}s(d)\theta_d^2\\
    &\geq L^2\sum_{d|P_s}\mu(d)\frac{s(d)}{d^2}- 2L\sum_{d|P_s}\frac{s(d)}{d} -\sum_{d|P_s}s(d)\\
    &= L^2\prod_{i=1}^s\left(1-\frac{2}{p_i^2}\right) - 2L\prod_{i=1}^s\left(1+ \frac{2}{p_i}\right) - 3^s.
\end{split}
\end{equation}
Also:
\begin{equation}\label{two points, expresion of B}
\begin{split}
    \sum_{i=s+1}^{\pi(L)}|I_{p_i}| &\leq 2\sum_{i=s+1}^{\pi(L)}\left\lceil\frac{L}{p_i}\right\rceil^2\leq 2\sum_{i=s+1}^{\pi(L)}\left(\frac{L}{p_i}+1\right)^2\\
    &\leq 2L^2\sum_{i=s+1}^{\pi(L)}\frac{1}{p_i^2}+ 4L\sum_{i=s+1}^{\pi(L)}\frac{1}{p_i} + 2\pi(L).
\end{split}
\end{equation}
By \eqref{two points visibility, proof first expresion}, \eqref{two points, expresion of A} and \eqref{two points, expresion of B},
\begin{equation}\label{two points, expresion of A-B}
\begin{split}
    \frac{|A\cap[1,L]^2|}{L^2}&\geq \prod_{i=1}^s\left(1-\frac{2}{p_i^2}\right)- 2\sum_{i=s+1}^{\pi(L)}\frac{1}{p_i^2} - \frac{2}{L}\cdot\prod_{i=1}^s\left(1+ \frac{2}{p_i}\right)\\
    &~~~- \frac{4}{L}\cdot\sum_{i=s+1}^{\pi(L)}\frac{1}{p_i}-\frac{3^s}{L^2}- \frac{2\pi(L)}{L^2}\\
    &= \prod_{i=1}^\infty\left(1-\frac{2}{p_i^2}\right)- f(s,L),
\end{split}
\end{equation}
where for $s, L\in \mathbf{N}$
\begin{equation*}
\begin{split}
    f(s,L)&= \prod_{i=1}^\infty\left(1-\frac{2}{p_i^2}\right)-\prod_{i=1}^s\left(1-\frac{2}{p_i^2}\right)+ 2\sum_{i=s+1}^{\pi(L)}\frac{1}{p_i^2}+\frac{2}{L}\cdot\prod_{i=1}^s\left(1+ \frac{2}{p_i}\right)\\
    &~~~+ \frac{4}{L}\cdot\sum_{i=s+1}^{\pi(L)}\frac{1}{p_i}+\frac{3^s}{L^2} + \frac{2\pi(L)}{L^2}.
\end{split}
\end{equation*}
We can easily see that the smallest $L$, for which the error is negative, is $L_0=7$.
The error at $L_0$ is
$$E= \frac{|A\cap[1,7]^2|}{7^2}-\prod_{i=1}^\infty\left(1-\frac{2}{p_i^2}\right)= -0.016\ldots. $$
As $f(s,L)$ decreases as a function of $L$, we only need to find $s$ and $L_1$ such that
\begin{equation}\label{two points, f(L,s) and E}
    f(s,L)\leq |E|,\qquad L\geq L_1.
\end{equation}
We easily check that \eqref{two points, f(L,s) and E} is true for $s=10$ and $L_1=5000$.
Hence
\begin{equation*}
    \begin{split}
  \SD(A)&= \min_{7\leq L\leq 5000}\frac{|A\cap[1,L]^2|}{L^2}= \frac{|A\cap[1,7]^2|}{7^2}= 0.306\ldots\\
  &< 0.322\ldots= \prod_{p\in \mathcal{P}}\left(1-\frac{2}{p^2}\right)= D(A). \end{split}
\end{equation*}
\end{proof}

\section{An Ergodic-theoretical Viewpoint}\label{section, ergodic}
The starting point of the paper, namely Dirichlet's and Lehmer's result about the asymptotic density of the set of lattice points visible from the origin, has been stated also in terms of probability~\cite{Nym} (see also \cite{Ben,CJ,BDS}).
Now, as there is no translation-invariant probability measure on $\Z^k$ (or even $\Z$), ``probability'' is to be understood here as the asymptotic probability of a uniformly random point in the cube $[0, L-1]^k$ being visible as $L \to \infty$. In this section, we show how the usage of the term probability may be made rigorous. Moreover, we explain how the basic result may be interpreted as an ergodic theorem. We will keep the required ergodic theory minimal.

Consider the function $f: \Z^k \to \{0,1\}$, given by
\begin{equation}\label{ergo- function f}
    f(n_1,\ldots , n_k)= \begin{cases} 1, &\qquad \gcd(n_1,\ldots,n_k)= 1,\\
    0, &\qquad \gcd(n_1,\ldots,n_k)> 1,
    \end{cases}
    \qquad   (n_1,\ldots,n_k) \in \Z^k.
\end{equation}
(It makes no difference for us, but we agree that $f$ vanishes at $(0,\ldots,0)$;~see \cite{HFSB}.)
Let $R_1, \ldots , R_k$ be the basic translation operators on $\Z^k$. Namely, we take:
\begin{equation*}\label{ergo expresion of T_i}
    R_i(n_1,\ldots, n_i,\ldots,n_k) = (n_1,\ldots, n_{i}+1, \ldots,n_k), \qquad 1\leq i\leq k,\, (n_1,\ldots,n_k) \in \Z^k.
\end{equation*}
With this notation, the probability that a uniformly random lattice point $(x_1,\ldots, x_k) \in [0,L-1]^k$ is visible from $(0,\ldots,0)\in \Z^k$ may be written in the form:
\begin{equation}\label{ergo, primiteve point as ergodic sum}
    P\left((x_1,\ldots,x_k)\,\mbox{is visible from}\, (0,\ldots,0)\right) = \frac{1}{L^k}\sum_{n_1,\ldots, n_k=0}^{L-1} f\left(R_1^{n_1}\cdots R_k^{n_k}(0,\ldots,0)\right).
\end{equation}
Now the expression on the right-hand side of \eqref{ergo, primiteve point as ergodic sum} looks like an ergodic average, for which it is natural to ask about the limit. However, the underlying space is not a probability space. Thus, we embed $\Z^k$ in a compact abelian group, and extend the translations $R_i$ in such a way that the right-hand side of \eqref{ergo, primiteve point as ergodic sum} will indeed be an ergodic average.

The group we take is $G^k$, where $G= \prod_{p\in \mathcal{P}}\Z/p\Z$ is the direct product of all cyclic groups of prime order. $G$ is a compact abelian group under the product topology. We describe the Haar measure $\mu$ on $G$, namely the (unique) probability measure on $G$, invariant under all translations of the group. (We refer to \cite{HPHN} for more details on the Haar measure.)
The measure of a cylindrical set $C= A_{1}\times A_{2}\times \cdots \times A_{n}\times \prod_{i \geq n+1}\Z/p_{i}\Z$, where $p_i$ is the $i$th prime and $A_j \subseteq \Z/p_j\Z$ for $1\leq j\leq n$, is:
\begin{equation}\label{ergo measure for cylinder}
    \mu(C)= \prod_{j=1}^{n}\frac{|A_j|}{p_j}.
\end{equation}
We mention that \eqref{ergo measure for cylinder} uniquely determines $\mu$ on the whole Borel field $\mathcal{B}(G)$ of~$G$~(see\cite{Kakut}). Moreover, $\mu$ is the Haar measure on $G$.
The Haar measure on $G^k$ is the $k$-fold product $\mu ^k= \mu \times \cdots \times \mu$, determined by~(see \cite[Sec. 5.3]{ETF})
$$\mu^k(B_1\times \cdots \times B_k) = \mu(B_1)\cdots \mu(B_k), \qquad B_1,\ldots , B_k \in \mathcal{B}(G).$$
Define a monomorphism $i_{\Z}: \Z\hookrightarrow G$ by
\begin{equation*}\label{ergo map i}
    i_{\Z}:n \to ( n  \ \textup{mod} \ p_1, n \ \textup{mod} \ p_2, \ldots ), \qquad n\in \Z.
\end{equation*}
The transformation $R_{\Z}: \Z\to \Z$, given by $R_{\Z}(n)= n+1$ for $n\in \Z$, may be extended to a transformation $R_{G}: G \to G$ by 
$$R_{G}: (x_{1}, x_{2},\ldots)= (x_{1}+1, x_{2}+1,\ldots),$$
and the diagram
\begin{equation}\label{commutative diagram 1}
 \begin{tikzcd}
\Z \arrow{r}{i_{\Z}} \arrow[swap]{d}{R_{\Z}} & G \arrow{d}{R_{G}} \\%
\Z \arrow{r}{i_\Z}& G
\end{tikzcd}
\end{equation}
is commutative. 

Now we generalize the concept of visibility to $G$ and $G^k$.
Two points $\textbf{x}= (x_1,x_2,\ldots)$ and $\textbf{y}=(y_1,y_2,\ldots)$ in $G$ are \textit{mutually visible} if $x_i\neq y_i$ for each $i$; otherwise, they are \textit{mutually invisible}. The probability that a $\mu$-random point $\textbf{x}\in G$ is visible from $\textbf{0}=(0,0,\ldots)\in G$ is the measure of the set $\prod_{p\in \mathcal{P}}(\Z/p\Z)^{*}$,
$$P(\textbf{x}\, \mbox{is visible from }\, \textbf{0}) = \prod_{p\in \mathcal{P}}\left(1-\frac{1}{p}\right)=0.$$
Two points $\overrightarrow{\textbf{x}}=(\textbf{x}_1,\ldots, \textbf{x}_k )$ and $ \overrightarrow{\textbf{y}}=(\textbf{y}_1,\ldots, \textbf{y}_k )$ in $G^k$, where $\textbf{x}_j = (x_{j1},x_{j2},\ldots)$ and $\textbf{y}_j = (y_{j1},y_{j2},\ldots)$ for $1\leq j\leq k$, are {\it mutually invisible} if $x_{ji}= y_{ji}$ for all $1\leq j\leq k$ and for some $i\in \N$; otherwise, they are {\it mutually visible}.
The probability that a $\mu^k$-random point $\overrightarrow{\textbf{x}} \in G^k$ is visible from $\overrightarrow{\textbf{0}}=(\textbf{0},\ldots,\textbf{0})$ is $\prod_{i=1}^{\infty}\left(1- \frac{1}{p_i^k}\right)=\frac{1}{\zeta(k)}$ for $k \geq 2$.
Let us denote 

\[
\textbf{e}_j=(\textbf{0},\ldots,\textbf{0},\here{\textbf{1}}{fromhere},\textbf{0},\ldots,\textbf{0}), \qquad 1\leq j\leq k,
\]
\begin{tikzpicture}[remember picture, overlay]
\node[font=\scriptsize, below right=12pt of fromhere] (tohere) {$j^{th}$-entry};
\draw[Stealth-] ([yshift=-4pt]fromhere.south) |- (tohere);
\end{tikzpicture}

\noindent where $\textbf{1}= (1,1,\ldots)$. Define the rotations
$R_{\textbf{e}_j} : G^k \to G^k$ (extending the translations $R_1,\ldots,R_k$ to $G^k$) by
$$R_{\textbf{e}_j}(\overrightarrow{\textbf{x}}) = \overrightarrow{\textbf{x}} + \textbf{e}_j,\qquad \overrightarrow{\textbf{x}} \in G^k, \, 1\leq j\leq k.$$
Analogously to \eqref{ergo, primiteve point as ergodic sum}, we may be interested in the behavior of the averages
\begin{equation}\label{ergo, visible point as ergodic sum}
     \frac{1}{L^k}\sum_{n_1,\ldots, n_k=0}^{L-1} f_{G^k}(R_{\textbf{e}_1}^{n_1}\cdots R_{\textbf{e}_k}^{n_k}\overrightarrow{\textbf{x}})
\end{equation}
as $L\to \infty$, where
\begin{equation*}\label{ergo defination of G^k}
    f_{G^k}(\overrightarrow{\textbf{x}})= \begin{cases} 1, & \qquad \overrightarrow{\textbf{x}}\, \mbox{ is visible from}\, \overrightarrow{\textbf{0}},\\
    0, & \qquad \mbox{otherwise}.
    \end{cases}
\end{equation*}

We recall several other basic definitions and results. (See~\cite{LIY} for more details.)
Let $X$ be a compact metric space, $\mathcal{A}$ its Borel $\sigma$-field, and $\nu$ a probability measure on $(X,\mathcal{A})$. Given commuting homeomorphisms $T_1, \ldots,T_k$ of $X$, we consider the $\Z^k$-action $T$ on $X$ defined by:
$$T^{\bar{n}}(x) = T_{1}^{n_1}\cdots T_{k}^{n_k}(x),\qquad x\in X, \, \bar{n} = (n_1,\ldots, n_k) \in \Z^k.$$
The measure $\nu$ is $T$-\textit{invariant} if 
$$\nu(T^{\bar{n}}(E))= \nu(E), \qquad E\in \mathcal{A},\, \bar{n}= (n_1,\ldots ,n_k)\in \Z^k.$$
A set $E\in \mathcal{A}$ is $T$-\textit{invariant} if $T^{\bar{n}}(E)= E$ for all $\bar{n}\in \Z^k$.
\begin{definition}\begin{enumerate}
 \item \emph{ A $T$-\textit{measure-preserving system} is a quadruple $(X,\mathcal{A}, \nu,T)$, where $\nu$ is $T$-invariant probability measure.}

\item\emph{ The system is \textit{ergodic} if, for every $T$-invariant set $E\in \mathcal{A}$, either $\nu(E) = 0 $ or $\nu(E)=1$}
\end{enumerate}
\end{definition}

\begin{definition}\emph{
An action $T$ is \textit{uniquely ergodic} if it admits a unique invariant measure on $X$.}
\end{definition}
Define a $\Z^k$-action $R$ on $G^k$ by: 
$$R^{\bar{n}} \overrightarrow{\textbf{x}}= R_{\textbf{e}_1}^{n_1} \cdots R_{\textbf{e}_k}^{n_k}\overrightarrow{\textbf{x}}, \qquad \bar{n}\in \Z^k, \, \overrightarrow{\textbf{x}}\in G^k.$$
Since the $R_{\textbf{e}_j}$-s are rotations of $G^k$, the action preserves $\mu^k$.
Note that the subgroup of $G^k$ generated by $\textbf{e}_j$, $1\leq j \leq k$, which is just the image of $\Z^k$ under the monomorphism~$i_{\Z}^k:\Z^k \to G^k$, is dense in $G^k$. Therefore, $R$ is invariant under all rotations of $G^k$, and hence is uniquely~ergodic by the uniqueness of the Haar measure.
Now the ergodicity of the action $R$ implies that (see \cite{LIY}):
\begin{equation}\label{ergo sum for L1 function}
    \frac{1}{L^k}\sum_{n_1,\ldots ,n_k =0}^{L-1}g\left(R^{\bar{n}} \overrightarrow{\textbf{x}}\right) \xrightarrow[L\to \infty]{\text{a.e.}} \int_{G^k} gd \mu^k,\qquad g\in L^1(G^k,\mathcal{B}^k(G),\mu^k).
\end{equation}
For continuous functions $g$, the convergence holds everywhere and is uniform~\cite[Proposition 2.8]{LIY}.

However, our function $f_{G^k}$ is not continuous. Indeed, take, for example,
$\overrightarrow{\textbf{x}}= (\textbf{x}_1,\ldots,\textbf{x}_k)$, $\overrightarrow{\textbf{y}}^n= (\textbf{y}_1^n,\ldots,\textbf{y}_k^n)$, with $\textbf{x}_j= \textbf{1}$ for $1\leq j\leq k$, and 
$$ y_{j i}^n = \begin{cases} 1,\qquad& 1\leq j\leq k, \, i\neq n\\
0,\qquad & 1\leq j \leq k,\, i=n.
\end{cases}
$$
Then $\overrightarrow{\textbf{y}}^n \xrightarrow[n\to \infty]{} \overrightarrow{\textbf{x}}$, yet $f_{G^k}(\overrightarrow{\textbf{y}}^n) =0$ for each $n$, while $f_{G^k}(\overrightarrow{\textbf{x}})= 1$. Thus, whereas the convergence in \eqref{ergo sum for L1 function} is guaranteed almost everywhere, it may not hold everywhere. It does hold for the point $\overrightarrow{\textbf{0}}$; this is just Dirichlet-Lehmer's result. However, there exist points for which the convergence in \eqref{ergo sum for L1 function} does not hold. The following example presents an ``extreme'' such point.
\begin{example}\emph{Let $M$ be any bijection from $\mathcal{P}$ to $(\N\cup\{0\})^k$,
and let $\overrightarrow{\textbf{x}}\in G^k$ be given by
$$\begin{pmatrix} x_{1i} \\ \vdots \\ x_{ki} \end{pmatrix}= -M(p_i) \, \textup{mod} \, p_i,\qquad i\in \N.$$
We claim that $f_{G^k}(R^{\bar{n}}{\overrightarrow{\textbf{x}}})=0$ for every $\bar{n} \in (\N\cup\{0\})^k$. In fact, given any $\bar{n}$, take the $i$ for which $M(p_i)= \bar{n}$. 
Then the $i$-th coordinate of $R^{\bar{n}}(\overrightarrow{\textbf{x}})$ is:
$$(-M(p_i) + \bar{n})\, \textup{mod} \, p_i\,=  \begin{pmatrix} 0 \\ \vdots \\ 0 \end{pmatrix}.$$
It follows that the left-hand side of \eqref{ergo sum for L1 function} is 0 for every $L$, while the right-hand side is $1/\zeta(k)$. Hence \eqref{ergo sum for L1 function} fails for $f_{G^k}$ and $\overrightarrow{\textbf{x}}$.}
\end{example}

We mention that, while in general there is no necessary relation between the ergodic averages and the almost everywhere limit at points where \eqref{ergo sum for L1 function} fails, in our case we have:
\begin{equation}\label{ergo sum is less than 1/zeta(k)}
   \limsup_{L\to \infty} \frac{1}{L^k}\sum_{n_1,\ldots ,n_k =0}^{L-1}f_{G^k}\left(R^{\bar{n}} \overrightarrow{\textbf{x}}\right) \leq \frac{1}{\zeta(k)},  \qquad \overrightarrow{\textbf{x}} \in G^k.
\end{equation}
In fact, define a sequence $\{f_{n,G^k}\}_{n\in \N}$ of continuous functions $f_{n,G^k}: G^k \to \R$ by:
\begin{equation*}
    f_{n,G^k}(\overrightarrow{\textbf{x}})= 
    \begin{cases}
    1, & \qquad  \textup{for all} \, 1\leq i\leq n,\,\textup{we have}\, x_{ji}\neq 0\, \textup{for some}\, j=j(i),\\
    0,& \qquad \textup{otherwise}.
    \end{cases}
\end{equation*}
Clearly, $f_{G^k}(\overrightarrow{\textbf{x}})\leq f_{n,G^k}(\overrightarrow{\textbf{x}})$ for every $n$ and $\overrightarrow{\textbf{x}}$. Hence, for all $\overrightarrow{\textbf{x}} \in G^k$ and $n\in \N$,
\begin{equation}\label{ergo limsum of sum}
\begin{split}
   \limsup_{L\to \infty} \frac{1}{L^k}\sum_{n_1,\ldots ,n_k =0}^{L-1}f_{G^k}\left(R^{\bar{n}} \overrightarrow{\textbf{x}}\right) &\leq\limsup_{L\to \infty}\frac{1}{L^k}\sum_{n_1,\ldots ,n_k =0}^{L-1}f_{n,G^k}\left(R^{\bar{n}} \overrightarrow{\textbf{x}}\right)\\
   &= \prod_{i=1}^n\left(1-1/p_i^k\right),
 \end{split}
\end{equation}
where the equality follows from the fact that $R$ is uniquely ergodic and the functions  $f_{n,G^k}$ are continuous. (In fact, the equality is trivial by periodicity.) Since the left-hand side of \eqref{ergo limsum of sum} is independent of $n$, this proves \eqref{ergo sum is less than 1/zeta(k)}.

We have seen that our action $R$ is uniquely ergodic, but $f_{G^k}$ is discontinuous. This explains why we have arbitrarily large cubes in $\Z^k$, for which the proportion of points visible from $\overrightarrow{\textbf{0}}$ is very far from the limit. In fact, as we have seen in Section \ref{sec, intro}, the proportion may well be $0$. The discontinuity of $f_{G^k}$ is due to its being sensitive to arbitrarily large primes. Consider a point $\overrightarrow{\textbf{x}} \in G^k$ that passes many ``invisibility tests'', namely satisfies $x_{ji}\neq0$ for all $1\leq i\leq n$ and some $j= j(i)$. If it satisfies this property for all primes, then $f_{G^k}(\overrightarrow{\textbf{x}})=1$; if it fails for but one prime, then $f_{G^k}(\overrightarrow{\textbf{x}})=0$.

Interestingly, if we replaced the binary visible-invisible ladder by a more refined measure of visibility, we would be able to enjoy the unique ergodicity property. Suppose, intuitively, that $f:G^k\to [0,1]$ still satisfies $f(\overrightarrow{\textbf{x}})=1$ for a visible point $\overrightarrow{\textbf{x}}$, but assumes a relatively large (small, respectively) value if $\overrightarrow{\textbf{x}}$ violates the visibility condition for few (many, respectively) primes.
Specifically, consider the functions $\Phi_{s,G^k}:G^k\to [0,1]$ for real $s>0$:
$$\Phi_{s,G^k}(\overrightarrow{\textbf{x}})= \prod_{i:\, x_{1i}=\cdots= x_{ki}= 0}\left(1-\frac{1}{p_i^s}\right),\qquad \overrightarrow{\textbf{x}}=(\textbf{x}_1,\ldots,\textbf{x}_k )\in G^k.$$
By \eqref{ergo sum for L1 function}:
\begin{equation}\label{ergo, phi(s,Gk)}
\begin{split}
    \frac{1}{L^k}\sum_{n_1,\ldots, n_k=0}^{L-1}\Phi_{s,G^k}(R^{\bar{n}}\overrightarrow{\textbf{x}}) &\xrightarrow[L\to \infty]{\textup{a.e.}}\int_{G^k}\Phi_{s,G^k}d\mu^k.\\
     \end{split}
\end{equation}
$G^k$ is the product of the finite groups $(\Z/p_i \Z)^k$, and the measure $\mu^k$ is the product of the Haar measures $\mu_i^k$ on these groups. The function $\Phi_{s,G^k}$ may be written as a product of simple functions $\Phi_{s,(\Z/p_i \Z)^k}$ on these groups. Hence we can calculate the integral explicitly to obtain:
\begin{equation}
\begin{split}
     \frac{1}{L^k}\sum_{n_1,\ldots, n_k=0}^{L-1}\Phi_{s,G^k}(R^{\bar{n}}\overrightarrow{\textbf{x}}) &\xrightarrow[L\to \infty]{\textup{a.e.}}\prod_{i=1}^{\infty}\int_{(\Z/p_i \Z)^k}\Phi_{s,(\Z/p_i \Z)^k}d \mu_i^k\\
     &=\prod_{i=1}^{\infty}\left[\left(1-\frac{1}{p_i^k}\right)\cdot 1 +\frac{1}{p_i^k}\left(1- \frac{1}{p_i^s}\right)\right]\\
     &= \prod_{i=1}^{\infty}\left(1-\frac{1}{p_i^{s+k}}\right)= \frac{1}{\zeta(s+k)}.
    \end{split}
\end{equation}
One checks easily that $\Phi_{s,G^k}$ is continuous for $s>1$. Hence, for such $s$, the convergence in \eqref{ergo, phi(s,Gk)} is everywhere and holds uniformly.
Note that, for $0<s\leq 1$, the function $\Phi_{s,G^k}$ is discontinuous. Indeed, take, say, $\overrightarrow{\textbf{x}}= (\textbf{x}_1,\ldots,\textbf{x}_k)$ with $\textbf{x}_j= \textbf{1}$ and $\overrightarrow{\textbf{y}}^n= (\textbf{y}_1^n,\ldots,\textbf{y}_k^n)$, defined by:
$$ y_{j i}^n = \begin{cases} 1,\qquad& 1\leq j\leq k, \, 1\leq i\leq n,\\
0,\qquad & 1\leq j \leq k,\, i>n.
\end{cases}
$$
Then $\overrightarrow{\textbf{y}}^n \xrightarrow[n\to \infty]{} \overrightarrow{\textbf{x}}$, yet $\Phi_{s,G^k}(\overrightarrow{\textbf{y}}^n) =0$ for each $n$, while $\Phi_{s,G^k}(\overrightarrow{\textbf{x}})= 1$. In the case $k=1$, the limit in \eqref{ergo, phi(s,Gk)} has been calculated in \cite{COHEN} and it was shown that the ergodic sum behaves like:
\begin{equation}\label{1 one ergodic average with error}
 \frac{1}{L}\sum_{n=0}^{L-1}\Phi_{s,G}(R_{G}^n\textbf{x}) = \frac{1}{\zeta(s+1)} +O\left(\frac{1}{L}\right), \qquad s>1.   
\end{equation}
The discussion above implies a multi-dimensional version of this formula. Moreover, due to the unique ergodicity, the averaging may start anywhere. Note, though, that our method does not yield an explicit error term as in \eqref{1 one ergodic average with error}.
\begin{prop}\label{ergodic average proposition}
Let $s>1$ and
$$B_i= [M_{1,i},M_{1,i}+ L_{1,i})\times \cdots \times[M_{k,i},M_{k,i}+ L_{k,i}), \qquad  i \in \N,$$
where $M_{1,i}\ldots M_{k,i}\in \N\cup \{0\}$ and $L_{j,i} \xrightarrow[i\to \infty]{} \infty$ for $1\leq j\leq k$. Then
\begin{equation*}
\begin{split}
    \frac{1}{|B_i|}\sum_{\bar{n}\in B_i}\Phi_{s,G^k}(R^{\bar{n}}\overrightarrow{\textbf{\em{x}}}) &\xrightarrow[i\to \infty]{}\int_{G^k}\Phi_{s,G^k}d\mu
    = \frac{1}{\zeta(s+k)}
    \end{split}
\end{equation*}
uniformly everywhere.
\end{prop}

\section{``Statistical'' Insights}\label{section, statistical}
In this section, we provide the results of several relevant computations performed using Mathematica. The first question we studied concerns the sign of the error $\frac{|V_k\cap [1,L|^k|}{L^k} - \frac{1}{\zeta(k)}$. As mentioned in the Section \ref{sec, intro}, the smallest integer, for which the error is negative for dimension $k=2$, is $L_0=820$. Thus, let $L$ be an \textit{exceptional} integer (for dimension $k$) if $\frac{|V_k\cap [1,L]^k|}{L^k} - \frac{1}{\zeta(k)}<0$. It turns out that, in the range $[1,10^4]$, there are but $18$ exceptional integer:
\begin{equation*}
    \begin{split}
        &820,1276,1422,1926,2080,2640,3186,3250,4446,\\
        4&720,4930,5370,6006,6546,7386,7476,9066,9276.
    \end{split}
\end{equation*}
The frequency of exceptional integers seems to be pretty fixed as we continue. There are 18237167 exceptional integers in the range $[1,10^{10}]$. Counting them in each of the 10 sub-intervals $(j\cdot10^9,(j+1)\cdot10^9]$, $0\leq j\leq 9$, we see that in each sub-interval, there are about one-tenth of them. In fact, the minimum is 1822954, attained at $(6\cdot10^9,7\cdot10^9]$ and the maximum is 1824549, attained at $(9\cdot10^9,10^{10}]$.

Another (hardly surprising) thing one may observe is that exceptional integers tend to have small remainders modulo small primes. In fact, all exceptional integers up to $5\cdot10^5$ are (i) $0\,\textup{mod}\, 2$, (ii) $0,1\,\textup{mod}\, 3$, (iii) $0,1,2\,\textup{mod}\, 5$, and (iv) $0,1,2,3,5\, \textup{mod}\, 7$.

For $k=3$, there are even (much) less exceptional integers. In the range $[1,10^6]$, there are six of them:
$$122760,169170,446370,689130,8134450,912990.$$
Similarly to the case $k=2$, the frequency of exceptional integers seems to be pretty fixed (but much lower) as we continue. 
There are 40815 exceptional integers in the range $[1,10^{10}]$. Counting them in each of the 10 sub-intervals $(j\cdot10^9,(j+1)\cdot10^9]$, $0\leq j\leq 9$, we see that the minimum is 4033, attained at $(4\cdot10^9,5\cdot10^9]$, and the maximum is 4123, attained at $(9\cdot10^9,10^{10}]$. We mention that the fact that for $k\geq 3$, the lower density (cf.\cite[p. 72]{FURSTEN}) of the set of exceptional integers is not 0 follows by getting into the proof of Petermann's results~\cite[p. 318]{PETER}, on which we have relied, in the proof of Theorem~\ref{visibility Omega +- theorem}. In fact, the set of exceptional integers contains an infinite arithmetic progression.

We have also looked at the error term for some sets $S$ of cardinality greater than~1. Interestingly, at least in the cases, we have checked, there is no strong tendency for the error to be positive (or negative either). We have considered the following examples.

\begin{example}\label{example, (0,1),(1,0)}
\emph{Let $S=\{(1,0),(0,1)\}\subseteq \mathbf{N}^2$ (as in Proposition \ref{Schni, proposition visibility from (0,1) and (1,0)}). 
In the range $[1,10^3]$ there are about 311 integers $L$ for which 
$$\frac{|V(S)\cap [1,L|^2|}{L^2}< \prod_{p\in \mathcal{P}}\left(1- \frac{2}{p^2}\right)= D(V(S)).$$
(The ``about'' in the last sentence is due to the fact that we have not bothered to check the status of integers $L$ for which $\frac{|V(S)\cap [1, L|^2|}{L^2}$ and the density $\prod_{p\in \mathcal{P}}\left(1- \frac{2}{p^2}\right)$ are very close.)
}
\end{example}

\begin{example}\label{example, (0,1),(1,0),(0,0)}
\emph{Let $S=\{(0,0),(1,0),(0,1)\}\subseteq \mathbf{N}^2$. Then the Schnirelmann density of $V(S)$ is strictly smaller than its regular density. In fact, in the range $[1,10^3]$ there are about 916 integers $L$ for which 
$$\frac{|V(S)\cap [1,L|^2|}{L^2}< \prod_{p\in \mathcal{P}}\left(1- \frac{3}{p^2}\right)= D(V(S)).$$
The Schnirelmann density of $V(S)$ seems to be
\begin{equation*}
\begin{split}
\min_{1\leq L\leq 10^3}\frac{|V(S)\cap [1,L|^2|}{L^2}&= \frac{|V(S)\cap [1,4|^2|}{4^2}=\frac{1}{16}= 0.0625\\
&< 0.1254\ldots= \prod_{p\in\mathcal{P}}\left(1- \frac{3}{p^2}\right)= D(V(S)),
\end{split}
\end{equation*}
but we have not verified it.}
\end{example}
\begin{example}\label{example, (1,0,0)(0,1,0),(0,0,1)}
\emph{Let $S=\{(1,0,0),(0,1,0),(0,0,1)\}\subseteq \mathbf{N}^3$. Again, the Schnirelmann density of $V(S)$ is strictly smaller than its regular density. In the range $[1,10^3]$ there are about 227 integers $L$ for which
$$\frac{|V(S)\cap [1,L|^3|}{L^3} <\prod_{p\in \mathcal{P}}\left(1- \frac{3}{p^3}\right)= D(V(S)).$$
The Schnirelmann density of $V(S)$ seems to be
\begin{equation*}
\begin{split}
\min_{1\leq L\leq 10^3}\frac{|V(S)\cap [1,L|^3|}{L^3}&= \frac{|V(S)\cap [1,16|^3|}{16^3}=\frac{2146}{4096}= 0.5239\ldots\\
&< 0.5345\ldots= \prod_{p\in\mathcal{P}}\left(1-\frac{3}{p^3}\right)= D(V(S).
\end{split}
\end{equation*}
}
\end{example}
\begin{example}\label{example, (0,0,0),(1,0,0)(0,1,0),(0,0,1)}
\emph{Let $S=\{(0,0,0),(1,0,0),(0,1,0),(0,0,1)\}\subseteq \mathbf{N}^3$. In the range $[1,10^3]$ there are about 279 integers $L$ for which
$$\frac{|V(S)\cap [1,L|^3|}{L^3}<\prod_{p\in \mathcal{P}}\left(1- \frac{4}{p^3}\right)= D(V(S)).$$
The Schnirelmann density of $V(S)$ seems to be
\begin{equation*}
\begin{split}
\min_{1\leq L\leq 10^3}\frac{|V(S)\cap[1,L|^3|}{L^3}&= \frac{|V(S)\cap[1,10|^3|}{10^3}=\frac{350}{1000}= 0.35\\
&< 0.404\ldots= \prod_{p\in \mathcal{P}}\left(1- \frac{4}{p^3}\right)= D(V(S).  \end{split}
\end{equation*}
}
\end{example}

One readily observes in all four examples that we have a much higher frequency of $L$-s, for which the proportion of visible points in $[1, L]^k$ is below $D(V(S))$, than was the case for visibility from the origin. Also, the Schnirelmann density is smaller than the regular density in each of these cases, and the difference $D(V(S))-\SD(V(S))$ is much larger than for $S=\{\bf{0}\}$.

We have checked a bunch of additional ``random'' sets $S$ in dimensions $k=2, 3, 4$. In all examples, we calculated the proportion of points visible from $S$ in the boxes $[1, L]^k$ for $1\leq L \leq 1000$ in dimensions $k=2,3$, and for $1\leq L\leq 200$ for $k=4$. Similarly to Examples \ref{example, (0,1),(1,0)}-\ref{example, (0,0,0),(1,0,0)(0,1,0),(0,0,1)}, it turns out that the proportions are smaller than the limiting value $D(V(S))$ for many values of $L$. We will refer to such $L$-s as having {\it bad visibility}. Denote by $L_{\textup{min}}$ the value of $L$ for which the proportion is minimal, namely at which the Schnirelmann density $\SD(V(S))$ is attained. (In principle, there may be more than one such $L$. We only found the smallest of these.) 
\begin{remark}\label{remark of exceptional integers}\emph{
We emphasize that our results may be {\bf not completely accurate} due to the following reasons:
\begin{enumerate}[label=(\alph*)]
    \item The infinite product $\prod_{p\in \mathcal{P}}\left(1- \frac{s(p)}{p^k}\right)$, giving $D(V(S))$, was calculated numerically, and so contains some error. Thus, values of $L$, for which $|V(S)\cap[1, L]^k|/L^k$ is very close to $D(V(S))$, may have been classified erroneously as having bad visibility or not.
    \item We have not gone in our tests far enough to make sure that there is no value of $L$ beyond the checked range, for which the proportion of visible points in $[1, L]^k$ is below the minimum we found. However, since the $L_{\textup{min}}$ we record is usually much smaller than the maximal value checked, we believe the results are mostly accurate. 
\end{enumerate}
}
\end{remark}

The full results are shown in Tables \ref{table:1}, \ref{table:2}, \ref{table:3}.

\begin{table}[!htbp]
\centering
\begin{tabular}{|  l | L{3cm} | R{2cm} | C{2cm} |C{2cm} |}
 \hline
 ~~~~~~~~~~~~~~~$S$ & Number of $L$-s with bad visibility in $[1,10^3]$&$L_{\textup{min}}$~~~~~~& $\SD(V(S))$&$D(V(S))$\\
 \hline
 (0,0),(1,0) & ~~~~~~~307    &10~~~~~~~~&0.29000\ldots& 0.32263\ldots\\\hline
 (0,0),(2,2) & ~~~~~~~~~27    &192~~~~~~~~&0.48348\ldots& 0.48396\ldots\\\hline
 (0,0),(6,0) & ~~~~~~~~~36    &156~~~~~~~~&0.55227\ldots& 0.55309\ldots\\\hline
 (0,0),(1,0),(2,3)  & ~~~~~~~~415    &16~~~~~~~~&0.10938\ldots& 0.12549\ldots\\\hline
 (0,0),(2,2),(3,4)&   ~~~~~~~~~19  & 810~~~~~~~~   &0.25073\ldots& 0.25097\ldots\\\hline
 (1,2),(4,5),(8,3)&   ~~~~~~~~990  & 18~~~~~~~~   &0.27160\ldots& 0.29280\ldots\\
 \hline
 (0,1),(2,2),(3,4)&   ~~~~~~~~486  & 18~~~~~~~~   &0.13889\ldots& 0.14640\ldots\\
 \hline
 (1,0),(2,2),(3,4),(4,3)&   ~~~~~~~~728  & 10~~~~~~~~  &0.10000\ldots& 0.11358\ldots\\
 \hline
  (0,0),(1,0),(0,1),(2,2)&   ~~~~~~~~723  & 4~~~~~~~~ &0.06250\ldots& 0.09465\ldots\\
 \hline
 (1,0),(0,1),(2,2),(3,4)&  ~~~~~~~~939  & 10~~~~~~~~ &0.09000\ldots& 0.11358\ldots\\
 \hline
 (0,0),(1,0),(0,1),(4,2)&  ~~~~~~~~920  & 16~~~~~~~~ &0.05469\ldots& 0.09465\ldots\\
 \hline
 (0,0),(1,0),(2,2),(4,2)&  ~~~~~~~~543  & 8~~~~~~~~&0.07812\ldots& 0.09465\ldots\\
 \hline
 (0,0),(1,0),(0,1)(2,2),(3,4)&  ~~~~~~~~762  & 4~~~~~~~~ &0.06250\ldots& 0.08542\ldots\\
 \hline
 (1,0),(0,1)(2,2),(3,4),(4,3)&  ~~~~~~~~992  & 10~~~~~~~~&0.08000\ldots& 0.10250\ldots\\
 \hline
 (0,0),(2,2)(3,2),(3,4),(4,3)&  ~~~~~~~~~~~8  & 4~~~~~~~~&0.06250\ldots& 0.06834\ldots\\
 \hline
 (0,0),(1,2)(2,1),(3,2),(4,3)&  ~~~~~~~~993  & 14~~~~~~~~&0.04592\ldots& 0.06834\ldots\\
 \hline
 (0,1),(1,2)(2,1),(3,2),(2,3)&  ~~~~~~~~988  & 17~~~~~~~~&0.11419\ldots& 0.13668\ldots\\
 \hline
\end{tabular}
 \caption{Statistics of sets of visible points for some 2-dimensional sets $S$ (see Remark \ref{remark of exceptional integers}).}
\label{table:1}
\end{table}
\begin{table}[!htbp]
    \centering
\begin{tabular}{|  l | L{3cm} | R{2cm} | C{2cm} |C{2cm} |}
 \hline
 ~~~~~~~~~~~~$S$ & Number of $L$-s with bad visibility in $[1,200]$&$L_{\textup{min}}$~~~~~~& $\SD(V(S))$&$D(V(S))$\\
 \hline
 (0,0,0),(1,0,0)  & ~~~~~~~~~91    &36~~~~~~~&0.67554\ldots& 0.67689\ldots\\\hline
 (1,0,0),(0,1,0)  & ~~~~~~~170   &16~~~~~~~&0.67188\ldots& 0.67689\ldots\\\hline
 (0,0,0),(2,2,2)  & ~~~~~~~~~28    &72~~~~~~~&0.78937\ldots& 0.78971\ldots\\
 \hline
  (0,0,0),(6,6,0)  & ~~~~~~~~~11    &6~~~~~~~& 0.81944\ldots& 0.82130\ldots\\\hline
  (6,6,0),(0,6,6)  & ~~~~~~~~~86    &6~~~~~~~& 0.80555\ldots& 0.82130\ldots\\\hline
 (1,1,0),(0,1,0),(0,0,1)  & ~~~~~~~~266    &16~~~~~~~& 0.52343\ldots& 0.53457\ldots\\
 \hline
 (0,0,0),(2,2,2),(3,3,0)  & ~~~~~~~~~33    &60~~~~~~~& 0.66794\ldots& 0.66820\ldots\\
 \hline
 (2,0,0),(2,2,2),(0,2,2)  & ~~~~~~~~~70    &72~~~~~~~& 0.74711\ldots& 0.74840\ldots\\
 \hline
 (0,0,0),(4,0,4),(6,0,6)  & ~~~~~~~~~34    &70~~~~~~~& 0.77897\ldots& 0.77958\ldots\\
 \hline
\end{tabular}
\caption{Statistics of sets of visible points for some 3-dimensional sets $S$ (see Remark \ref{remark of exceptional integers}).}
\label{table:2}
\end{table}
\begin{table}[!htbp] 
    \centering
\begin{tabular}{|  l | L{3cm} | R{2cm} | C{2cm} |C{2cm} |}
 \hline
 ~~~~~~~~~~~~~~~~$S$ & Number of $L$-s with bad visibility in $[1,200]$&$L_{\textup{min}}$~~~~~~& $\SD(V(S))$&$D(V(S))$\\
 \hline
 (1,1,0,0),(0,1,1,1)  & ~~~~~~~~161    &12~~~~~~~&0.82424\ldots& 0.84973\ldots\\
 \hline
 (0,0,0,0),(2,2,2,2)  & ~~~~~~~~~~~9    &12~~~~~~~&0.90987\ldots& 0.91043\ldots\\
 \hline
 (1,1,0,0),(0,0,1,1)  & ~~~~~~~~~~90    &16~~~~~~~&0.84634\ldots& 0.84974\ldots\\
 \hline
 (0,0,1,0),(1,1,0,0),(1,0,0,1)  & ~~~~~~~~~~46  &4~~~~~~~&0.77344\ldots& 0.77738\ldots\\\hline
 (0,0,0,0),(1,1,0,0),(0,0,1,1)  & ~~~~~~~~~~33&16~~~~~~~&0.77556\ldots& 0.77738\ldots\\\hline
\end{tabular}
\caption{Statistics of sets of visible points for some 4-dimensional sets $S$ (see Remark \ref{remark of exceptional integers}).}
\label{table:3}
\end{table}

In view of our observations in the beginning of this section and the discussion in Section \ref{section, ergodic} we pose our final
\begin{question}\emph{ Is it true that, for every $k$ and finite $S\subset \Z^k$, the set of integers $L$ with bad visibility has an asymptotic density? Is it true, moreover, that this set has a Banach density? (For the notion of Banach density, see, for example,~\cite[p. 72]{FURSTEN}.)}
\end{question}
\section{Visibile Lattice Points in Discs}\label{Section on visibility in discs}
Denote by
$$A_k(x)=|\{(n_1,\ldots,n_k)\in \Z^k:\,  n_1^2+ \ldots + n_k^2\leq x\}|,\qquad k=2,3,\ldots, \qquad x\geq 0,$$
the number of lattice points in the $k$-dimensional disc $D_k(0,\sqrt{x})$ of radius $\sqrt{x}$, centered at the origin. $A_k(x)$ is very close to the volume of $D_k(0,\sqrt{x})$ for large $x$. More precisely, put
$$P_k(x)=A_k(x)- \textup{Vol}(D_k(0,\sqrt{x})), \qquad k=2,3,\ldots,\qquad x\geq 0.$$
The well-known Gauss circle problem (cf.\cite{Ivic}) is to estimate $P_2(x)$. For general $k\geq 2$, the estimation of $P_k(x)$ is known as the generalized Gauss circle problem (see \cite{Thomas}). Gauss showed, using a simple observation, that $P_2(x)= O(\sqrt{x})$. The error was improved by several authors (cf. \cite{Sierpinski, Kolesnic, Huxley1}), and the best estimate currently seems to be due to Bourgain and Watt \cite[Theorem 2]{Watt}:
$$P_2(x)= O\left(x^{517/1648+ \varepsilon}\right), \qquad \varepsilon >0.$$
In the other direction, Hardy \cite{Hardy1} showed that 
$$P_2(x)= \Omega_{+}(x^{1/4}), \qquad P_2(x)= \Omega_{-}(x^{1/4}(\log x)^{1/4}).$$
The $\Omega$-results were improved by Corr{\'a}di-K{\'a}tai \cite{Katai}, Hafner \cite{Hafner}, and Soundararajann \cite[Theorem 1]{Soundar}.
Similar results to those on $D_2(0,\sqrt{x})$ have been obtained for compact convex regions $\mathcal{D}\subseteq \R^2$ containing the origin, whose boundary satisfies some smoothness conditions (see, for example, \cite{Huxley2, Nowak2, Nowak3}). Huxley \cite[Theorem 5]{Huxley2} showed that the estimation from above of the analog of $P_2(x)$ for such $\mathcal{D}$ is not more difficult than that of $P_2(x)$ (see also \cite{Nowak}).

For the estimation of $P_k(x)$ in dimensions 3 and 4, we refer to \cite{Landau, Brown, Tsang, Walfisz1, Adhikari2}. 
For dimension $k\geq 5$, the situation becomes simpler (cf. \cite{Kratze}):
$$P_k(x)= O(x^{k/2-1}), \qquad P_k(x)= \Omega(x^{k/2-1}), 
\qquad k\geq 5.$$

Denote by
$$V_k'(x)=|\{(n_1,\ldots,n_k)\in \Z^k:\,  n_1^2+ \ldots + n_k^2\leq x,\, (n_1,\ldots,n_k)=1\}|,\quad k=2,3,\ldots,\,x\geq 0,$$
the number of lattice points visible from the origin in $D_k(0,\sqrt{x})$. $V_k'(x)$ is very close to $1/\zeta(k)\cdot \textup{Vol}(D_k(0,\sqrt{x}))$ for large $x$. Put
\begin{equation*}
E_k'(x)=V_k'(x)- \frac{1}{\zeta(k)}\cdot \textup{Vol}(D_k(0,\sqrt{x})), \qquad k=2,3,\ldots,\qquad x\geq 0.    
\end{equation*}
Huxley and Nowak \cite{Nowak} show that 
$$E_2'(x)= O(x^{1/2}\exp(-c(\log x)^{3/5}(\log\log x)^{-1/5})),$$
for some constant $c>0$, and under the Riemann Hypothesis
$E_2'(x)=O(x^{5/12+ \varepsilon})$ for arbitrary fixed $\varepsilon>0$. Currently, the best-known bound is due to Wu \cite{Wu}:
$$E_2'(x)= O\left(x^{221/608+ \varepsilon}\right), \qquad \varepsilon >0.$$
$\Omega$-results for $E_2'$ and $E_3'$ have also been obtained \cite{Nowak4, Fernando}. (For more details on visible lattice points in planar domains, see \cite{Barany, Moroz1, Nowak6, Zhai, Cao, Baker, Hensley}.)

In the spirit of this paper, it seems more appropriate to compare $V_k'(x)$ not with $1/\zeta(k)\cdot\textup{Vol}(D(0,\sqrt{x}))$ but rather with $1/\zeta(k)\cdot A_k(x)$. We wanted to check computationally whether, for visibility from the origin, we again have the phenomenon whereby, for most discs, the relative density of visible points within the disc exceeds the asymptotic density. The short answer is negative.
Similarly to the terminology from Section \ref{section, statistical}, let a positive integer $n$ be ``exceptional'' in this section (for dimension $k$) if 
$$V_k'(n)< A_k(n)/\zeta(k).$$
It turns out that, unlike the case of visible points in cubes $[1, L]^k$, exceptional integers for discs are not exceptional at all. In fact, for $k=2$, in the range $[1,10^9]$, there are 474072530 exceptional integers. Counting them in each of the ten sub-intervals $(j\cdot10^8,(j+1)\cdot10^8]$, $0\leq j\leq 9$, we find that the minimum is 20862314, attained at $(7\cdot 10^8, 8\cdot 10^8]$, and the maximum is 63524202, attained at $(5\cdot 10^8, 6\cdot 10^8]$. For $k=3$, in the range $[1,10^6]$, there are 500724 exceptional integers. Counting them in each of the ten sub-intervals $(j\cdot10^5,(j+1)\cdot10^5]$, $0\leq j\leq 9$, we find that the minimum is 46484, attained at $(6\cdot10^5, 7\cdot10^5]$, and the maximum is 51656, attained at $(7\cdot10^5, 8\cdot10^5]$. For $k=4$, in the range $[1, 10^6]$, there are 500220 exceptional integers. Counting them in each of the ten sub-intervals $(j\cdot10^5,(j+1)\cdot10^5]$, $0\leq j\leq 9$, we find that the minimum is 49982, attained at $(9\cdot10^5, 10^6]$, and the maximum is 50092, attained at $(8\cdot10^5, 9\cdot10^5]$.

Recall that $V_k$ is the set of lattice points visible from the origin. The Schnirelmann density of $V_k$ was defined as the minimum over $L$ of the relative density $|V_k\cap[1, L]^k|/L^k$ of the set of visible points within cubes $[1, L]^k$. We may consider the analogous ratio when we go over discs $D_k(0,\sqrt{x})$. Namely, let
$$\SD'(V_k)= \inf_{n\in N} \frac{V_k'(n)}{A_k(n)}, \qquad n= 1,2,\ldots,\qquad (k=1,2\ldots).$$
(As mentioned above, the origin is considered as invisible.) Our calculations hint that
$$\SD'(V_2)= \frac{V_2'(9)}{A_2(9)}= \frac{16}{29}=0.552< 0.608= 1/\zeta(2),$$
$$\SD'(V_3)= \frac{V_3'(4)}{A_3(4)}= \frac{26}{33}=0.788< 0.832= 1/\zeta(3),$$
$$\SD'(V_4)= \frac{V_4'(1)}{A_4(1)}=\frac{8}{9}= 0.889< 0.924= 1/\zeta(4).$$
\newpage
\bibliographystyle{plain}

\end{document}